\documentclass{tran-l}

\usepackage{amsmath}
\usepackage{amssymb}
\usepackage{amsfonts}
\usepackage{latexsym, amsthm, mathrsfs}
\usepackage[german, english]{babel}

\newcommand{\om}{\omega}

\newcommand{\vp}{\varphi}

\newcommand{\sse}{\subseteq}
\newcommand{\contains}{\supseteq}

\DeclareMathOperator{\depth}{depth}

\DeclareMathOperator{\E}{E}

\DeclareMathOperator{\R}{R}
\DeclareMathOperator{\Ext}{Ext}
\DeclareMathOperator{\RP}{RP}
\DeclareMathOperator{\HOD}{HOD}

\newcommand{\rgl}{\rangle}
\newcommand{\lgl}{\langle}

\newcommand{\re}{\restriction}

\newcommand{\bT}{\mathbb{T}}

\newcommand{\bP}{\mathbb{P}}
\newcommand{\bR}{\mathbb{R}}

\newcommand{\bN}{\mathbb{N}}

\newcommand{\ra}{\rightarrow}

\newcommand{\Erdos}{Erd\H{o}s}

\newcommand{\Juhasz}{Juh{\'{a}}sz}
\newcommand{\Rodl}{R{\"{o}}dl}

\begin{document}

\newtheorem{thm}{Theorem}[section]
\newtheorem{prop}[thm]{Proposition}
\newtheorem{lem}[thm]{Lemma}
\newtheorem{cor}[thm]{Corollary}
\newtheorem{fact}[thm]{Fact}
\newtheorem{claim}[thm]{Claim}
\newtheorem*{thmMT}{Main Theorem}
\newtheorem*{thmnonumber}{Theorem}
\newtheorem*{mainclaim}{Main Claim}
\newtheorem*{claim1}{Claim 1}
\newtheorem*{claim2}{Claim 2}
\newtheorem*{claim3}{Claim 3}
\newtheorem*{claim4}{Claim 4}
\newtheorem*{claim5}{Claim 5}
\newtheorem*{claim6}{Claim 6}
\newtheorem*{claim7}{Claim 7}
\newtheorem*{claim8}{Claim 8}
\newtheorem*{claim9}{Claim 9}
\newtheorem*{claim10}{Claim 10}
\newtheorem*{claim11}{Claim 11}
\newtheorem*{claim12}{Claim 12}
\newtheorem*{claim13}{Claim 13}
\newtheorem*{claim14}{Claim 14}
\newtheorem*{claim15}{Claim 15}
\newtheorem*{claimA}{Claim A}
\newtheorem*{claimC}{Claim C}
\newtheorem*{claimD}{Claim D}
\newtheorem*{factD}{Fact D}
\newtheorem{claimn}{Claim}

\theoremstyle{definition}
\newtheorem{defn}[thm]{Definition}
\newtheorem{example}[thm]{Example}
\newtheorem{conj}[thm]{Conjecture}
\newtheorem{prob}[thm]{Problem}
\newtheorem{examples}[thm]{Examples}
\newtheorem{question}[thm]{Question}
\newtheorem{problem}[thm]{Problem}
\newtheorem{openproblems}[thm]{Open Problems}
\newtheorem{openproblem}[thm]{Open Problem}
\newtheorem{conjecture}[thm]{Conjecture}
\newtheorem*{problem1}{Problem 1}
\newtheorem*{problem2}{Problem 2}
\newtheorem*{notn}{Notation}

\theoremstyle{remark}
\newtheorem{rem}[thm]{Remark}
\newtheorem{rems}[thm]{Remarks}
\newtheorem*{ack}{Acknowledgments}
\newtheorem{note}{Note}
\newtheorem{subclaim}{Subclaim}
\newtheorem*{subclaimn}{Subclaim}
\newtheorem*{subclaim1}{Subclaim (i)}
\newtheorem*{subclaim2}{Subclaim (ii)}
\newtheorem*{subclaim3}{Subclaim (iii)}
\newtheorem*{subclaim4}{Subclaim (iv)}
\newtheorem*{case1}{Case 1}
\newtheorem*{case2}{Case 2}
\newtheorem*{case3}{Case 3}
\newtheorem*{case4}{Case 4}
\newtheorem*{case5}{Case 5}
\newtheorem*{case6}{Case 6}
\newtheorem*{case7}{Case 7}
\newtheorem*{case8}{Case 8}
\newtheorem*{case9}{Case 9}
\newtheorem*{case10}{Case 10}
\newtheorem*{case11}{Case 11}
\newtheorem*{case12}{Case 12}
\newtheorem*{case13}{Case 13}
\newtheorem*{case14}{Case 14}
\newtheorem*{case15}{Case 15}
\newtheorem*{subcasei}{Subcase (i)}
\newtheorem*{subcaseii}{Subcase (ii)}
\newtheorem*{subcaseiii}{Subcase (iii)}
\newtheorem*{subcaseiv}{Subcase (iv)}
\newtheorem*{subcasev}{Subcase (v)}
\newtheorem*{subcasevi}{Subcase (vi)}
\newtheorem*{subsubcasea}{Sub-subcase (a)}
\newtheorem*{subsubcaseb}{Sub-subcase (b)}

\numberwithin{equation}{section}

\title[A Ramsey-Classification Theorem]{A Ramsey-Classification Theorem and its Application in the Tukey Theory of
Ultrafilters}

\author{Natasha Dobrinen}
\address{University of Denver\\
Department of Mathematics, 2360 S Gaylord St, Denver, CO 80208, USA}
\email{natasha.dobrinen@du.edu}
\thanks{The first author was supported by an Association for Women in Mathematics - National Science Foundation Mentoring Travel Grant
 and a University of Denver Faculty Research Fund Grant}

\author{Stevo Todorcevic}
\address{Department of Mathematics\\
University of Toronto\\
Toronto\\ Canada\\ M5S 2E4}
\email{stevo@math.toronto.edu}
\address{Institut de Mathematiques de Jussieu\\
CNRS - UMR 7056\\
75205 Paris\\
France}
\email{stevo@math.jussieu.fr}
\thanks{The second author was supported by grants from NSERC and CNRS}


\date{}

\begin{abstract}
Motivated by a Tukey classification problem we develop here a new
topological Ramsey space $\mathcal{R}_1$ that in its complexity
comes immediately after the classical  Ellentuck
space \cite{MR0349393}. Associated with $\mathcal{R}_1$ is an
 ultrafilter $\mathcal{U}_1$  which is  weakly Ramsey but not Ramsey.
We prove a canonization theorem for equivalence relations on
fronts on $\mathcal{R}_1$. This is analogous to the  Pudlak-\Rodl\
Theorem canonizing  equivalence relations on barriers on the
Ellentuck space. We then apply our canonization theorem to
completely classify all Rudin-Keisler equivalence classes of
ultrafilters which are Tukey reducible to $\mathcal{U}_1$: Every
ultrafilter which is Tukey reducible to $\mathcal{U}_1$ is
isomorphic to a countable iteration of Fubini products of
ultrafilters from among a fixed countable collection of
ultrafilters. Moreover, we show that there is exactly one Tukey
type of nonprincipal ultrafilters strictly below that of
$\mathcal{U}_1$, namely the Tukey type of a Ramsey ultrafilter.
\end{abstract}

\maketitle

\section{Overview}\label{sec.overview}

Motivated by a Tukey classification problem and inspired by work of Laflamme in \cite{Laflamme89} and the second author in \cite{Raghavan/Todorcevic11},
we build a new topological Ramsey space $\mathcal{R}_1$.
This space, $\mathcal{R}_1$,  is minimal in complexity above the classical Ellentuck space,
the Ellentuck space being obtained as the projection of $\mathcal{R}_1$ via
  a fixed finite-to-one map.
Every topological Ramsey space has  notions of finite approximations, fronts, and barriers.
In Theorem
  \ref{thm.original}, we prove that
for each $n$,  there is a finite collection of canonical equivalence relations for uniform barriers on $\mathcal{R}_1$ of rank $n$.
That is,  we  show that given  $n$, for any uniform barrier $\mathcal{B}$ on $\mathcal{R}_1$ of finite rank  and 
any equivalence relation $\E$ on $\mathcal{B}$,
there is an $X\in\mathcal{R}_1$ such that $\E$ restricted to the members of  $\mathcal{B}$ coming from within $X$ is exactly one of the canonical equivalence relations.
The canonical equivalence relations are represented
by a certain collection of finite trees.
  This  generalizes the \Erdos-Rado Theorem  for barriers of the form $[\bN]^n$.
In the  main theorem of this paper,
Theorem \ref{thm.PRR(1)},
we prove a new Ramsey-classification theorem for all barriers on the topological Ramsey space $\mathcal{R}_1$: 
We prove that for any barrier $\mathcal{B}$ on $\mathcal{R}_1$ and any equivalence relation  on $\mathcal{B}$,
there is an  inner Sperner map which canonizes the equivalence relation.
This generalizes the Pudlak-\Rodl\ Theorem for  barriers on the Ellentuck space.
These classification theorems were motivated by the following.

Recently the second author (see Theorem 24 in
\cite{Raghavan/Todorcevic11}) has made a connection between the
Ramsey-classification theory (also known as the canonical Ramsey
theory) and the Tukey classification theory of ultrafilters  on
$\omega$.
 More precisely, he showed that selective ultrafilters
realize  minimal Tukey types in the class of all ultrafilters on
$\omega$ by applying the Pudlak-R\"{o}dl Ramsey classification
result to a given cofinal map from a selective ultrafilter into
any other ultrafilter on $\omega,$ a map which, on the basis of
our previous paper \cite{Dobrinen/Todorcevic10}, he could assume
to be continuous. Recall that the notion of a selective
ultrafilter is closely tied to the Ellentuck space on the family of
all infinite subsets of $\omega,$ or rather the one-dimensional
version of the pigeon-hole principle on which the Ellentuck space is
based, the principle stating that an arbitrary
$f:\omega\rightarrow\omega$ is either constant or is one-to-one on
an infinite subset of $\omega.$ Thus an ultrafilter $\mathcal{U}$
on $\omega$ is \emph{selective} if for every map
$f:\omega\rightarrow\omega$ there is an $X\in \mathcal{U}$ such that
$f$ is either constant or one-to-one on $\mathcal{U}.$ Since
essentially any other topological Ramsey space has it own notion
of a selective ultrafilter living on the set of its
$1$-approximations (see \cite{MR2330595}), the argument for
Theorem 24 in \cite{Raghavan/Todorcevic11} is so general that it
will give analogous Tukey-classification results for all
ultrafilters of this sort provided, of course, that we have the
analogues of the Pudlak-R\"{o}dl Ramsey-classification result for
the corresponding topological Ramsey spaces.

This paper is our
first step towards a research in this direction. 
In particular,
inspired by work of Laflamme \cite{Laflamme89}, we build a
topological Ramsey space $\mathcal{R}_1$, so that the ultrafilter associated with $\mathcal{R}_1$ is isomorphic to the  ultrafilter $\mathcal{U}_1$
 forced by Laflamme. 
In \cite{Laflamme89},
Laflamme forced an ultrafilter, $\mathcal{U}_1$, which is  weakly
Ramsey but not Ramsey, and satisfies additional partition
properties. Moreover, he showed that $\mathcal{U}_1$ has complete
combinatorics over the Solovay model. By work of Blass in
\cite{Blass74}, $\mathcal{U}_1$ has only one non-trivial
Rudin-Keisler equivalence class of ultrafilters strictly below it,
namely that of the projection of $\mathcal{U}_1$ to a Ramsey
ultrafilter denoted $\mathcal{U}_0$. Thus, the Rudin-Keisler
classes of nonprincipal ultrafilters which are Rudin-Keisler
reducible to $\mathcal{U}_1$ forms a chain of length 2.  
At this
point it is instructive to recall another result of the second
author (see Theorem 4.4 in \cite{MR1644345}) stating that assuming
sufficiently strong large cardinal axioms \emph{every} selective
ultrafilter is generic over $L(\mathbb{R})$ for the partial order
of infinite subsets of $\omega,$ and the same argument applies for
any other ultrafilter that is selective relative any other
topological Ramsey space (see \cite{MR2330595}). Since, as it is
well-known, assuming large cardinals, the theory of
$L(\mathbb{R})$ cannot be changed by forcing, this gives another
perspective to the notion of `complete combinatorics' of Blass
and Laflamme.

One line of motivation for the work in this paper was to find  the structure of the Tukey types of nonprincipal ultrafilters Tukey reducible to $\mathcal{U}_1$.
We show in Theorem \ref{cor.1Tpred}
that, in fact, the  only  Tukey type of nonprincipal  ultrafilters  strictly below that of $\mathcal{U}_1$ is the Tukey type of $\mathcal{U}_0$.
Thus, the structure of the Tukey types below $\mathcal{U}_1$ is the same as the structure of the Rudin-Keisler equivalence classes below $\mathcal{U}_1$.
The second and stronger motivation for this work  was to find  a canonization theorem for equivalence relations on  fronts on  $\mathcal{R}_1$, and to apply it to obtain a
 finer result than  Theorem \ref{cor.1Tpred}.
 Applying Theorem \ref{thm.PRR(1)},
we  completely classify all Rudin-Keisler classes of ultrafilters which are contained in the Tukey types of $\mathcal{U}_1$ and $\mathcal{U}_0$ in Theorem \ref{thm.TukeyU_1}.
This extends the second author's
Theorem 24 in
\cite{Raghavan/Todorcevic11}, classifying the Rudin-Keisler classes within the  Tukey type of a Ramsey ultrafilter.
We remark that the fact that $\mathcal{R}_1$ is a topological Ramsey space is essential to the proof of Theorem \ref{thm.TukeyU_1}, and that forcing alone is not sufficient to obtain our result.

%**********************************************************************************************************************************************************

\section{Introduction and Background}\label{sec.intro}

We now introduce this work including the necessary background and notions.
Let $\mathcal{U}$ be an ultrafilter on a countable base set.
A subset $\mathcal{B}$ of an ultrafilter $\mathcal{U}$ is called {\em cofinal} if it is a base for the ultrafilter $\mathcal{U}$; that is, if for each $U\in\mathcal{U}$ there is an $X\in\mathcal{B}$ such that $X\sse U$.
Given ultrafilters $\mathcal{U},\mathcal{V}$,
we say that a function $g:\mathcal{U}\ra \mathcal{V}$
 is {\em cofinal} if the image of each cofinal subset of $\mathcal{U}$ is cofinal in $\mathcal{V}$.
We say that $\mathcal{V}$ is {\em Tukey reducible} to $\mathcal{U}$, and write $\mathcal{V}\le_T \mathcal{U}$, if there is a cofinal map from $\mathcal{U}$ into $\mathcal{V}$.
If both $\mathcal{V}\le_T \mathcal{U}$ and $\mathcal{U}\le_T \mathcal{V}$, then we write $\mathcal{U}\equiv_T \mathcal{V}$ and say that $\mathcal{U}$ and $\mathcal{V}$ are Tukey equivalent.
$\equiv_T$ is an equivalence relation, and  $\le_T$ on the equivalence classes forms a partial ordering.
The equivalence classes are called {\em Tukey types}.

A cofinal map $g:\mathcal{U}\ra\mathcal{V}$ is called {\em monotone} if whenever $U\contains U'$ are elements of $\mathcal{U}$, we have $g(U)\contains g(U')$.
It is a fact that $\mathcal{U}\ge_T\mathcal{V}$ if and only if there is a monotone cofinal map witnessing this.
(See Fact 6 in \cite{Dobrinen/Todorcevic10}.)
Thus, we need only consider monotone cofinal maps.
We point out that  $\mathcal{U}\ge_T\mathcal{V}$ if and only if there are cofinal subsets $\mathcal{B}\sse\mathcal{U}$ and $\mathcal{C}\sse\mathcal{V}$ and a  map $g:\mathcal{B}\ra\mathcal{C}$ which is a cofinal map from $\mathcal{B}$ into $\mathcal{C}$.
This fact will be used throughout this section.

We remind the reader of the Rudin-Keisler reducibility relation.
Given two ultrafilters $\mathcal{U}$ and $\mathcal{V}$, we say that $\mathcal{U}\le_{RK}\mathcal{V}$ if and only if there is a function $f:\om\ra\om$ such that $\mathcal{U}=f(\mathcal{V})$, where
\begin{equation}
f(\mathcal{V})=\lgl\{f(U):U\in\mathcal{U}\}\rgl.
\end{equation}
Recall that $\mathcal{U}\equiv_{RK}\mathcal{V}$ if and only if $\mathcal{U}$ and $\mathcal{V}$ are isomorphic.

Tukey reducibility on ultrafilters generalizes  Rudin-Keisler reducibility in that $\mathcal{U}\ge_{RK} \mathcal{V}$ implies that $\mathcal{U}\ge_T\mathcal{V}$.
The converse does not hold.
There are $2^{\mathfrak{c}}$ many ultrafilters in the top Tukey type (see \Juhasz \cite{Juhasz66} and Isbell \cite{Isbell65}), whereas every Rudin-Keisler equivalence class has cardinality  $\mathfrak{c}$.

However, it is consistent that there are ultrafilters with Tukey type of cardinality $\mathfrak{c}$.
We remind the reader of the following special kinds of ultrafilters.

\begin{defn}[\cite{Bartoszynski/JudahBK}]
Let $\mathcal{U}$ be an ultrafilter on $\om$.
\begin{enumerate}
\item
$\mathcal{U}$ is {\em Ramsey} if for each coloring $c:[\om]^2\ra 2$, there is a $U\in\mathcal{U}$ such that $U$ is homogeneous, meaning $|c''[U]^2|=1$.
\item
$\mathcal{U}$ is {\em weakly Ramsey} if for each coloring $c:[\om]^2\ra 3$, there is a $U\in\mathcal{U}$ such that $|c''[U]^2|\le 2$.
\item
$\mathcal{U}$ is a {\em p-point} if for each decreasing sequence $U_0\contains U_1\contains\dots$ of elements of $\mathcal{U}$, there is an $X\in\mathcal{U}$ such that $|X\setminus U_n|<\om$, for each $n<\om$.
\item
$\mathcal{U}$ is  {\em rapid} if for each function $f:\om\ra\om$, there is an $X\in\mathcal{U}$ such that $|X\cap f(n)|\le n$ for each $n<\om$.
\end{enumerate}
\end{defn}

Every Ramsey ultrafilter is weakly Ramsey,  which is in turn both a  p-point and rapid.
All of these sorts of ultrafilters are consistent with ZFC, and exist in every model of CH or MA.
Ramsey ultrafilters are also called {\em selective}, and the property of being Ramsey is equivalent to the following property:
For each decreasing sequence $U_0\contains U_1\contains\dots$ of members of $\mathcal{U}$, there is an $X\in\mathcal{U}$ such that  for each $n<\om$, $X\sse^* U_n$ and moreover $|X\cap(U_{n+1}\setminus U_n)|\le 1$.

Any subset of $\mathcal{P}(\om)$ is a   topological space, with the subspace topology inherited from the Cantor space.
Thus, given any
$\mathcal{B},\mathcal{C}\sse\mathcal{P}(\om)$,
a function $g:\mathcal{B}\ra\mathcal{C}$ is continuous if it is continuous with respect to the subspace topologies on $\mathcal{B}$ and $\mathcal{C}$.
Equivalently,
 a function $g:\mathcal{B}\ra\mathcal{C}$ is continuous if for each sequence $(X_n)_{n<\om}\sse\mathcal{B}$ which converges to some $X\in\mathcal{B}$,
the sequence $(g(X_n))_{n<\om}$ converges to $g(X)$, meaning that  for all $k$ there is an $n_k$ such that for all $n\ge n_k$,
$g(X_n)\cap k=g(X)\cap k$.
For any ultrafilter $\mathcal{V}$, cofinal $\mathcal{C}\sse\mathcal{V}$, and  $X\in\mathcal{V}$,  we use $\mathcal{C}\re X$ to denote $\{Y\in\mathcal{C}:Y\sse X\}$.
Note that $\mathcal{C}\re X$ is a cofinal subset of $\mathcal{V}$ and hence is a filter base for $\mathcal{V}$.
Thus, $(\mathcal{U},\contains)\equiv_T(\mathcal{C}\re X,\contains)$.

The authors  proved in Theorem 20 of \cite{Dobrinen/Todorcevic10} that if $\mathcal{U}$ is a p-point and $\mathcal{U}\ge_T\mathcal{W}$,
then there is a continuous monotone cofinal map witnessing this.

\begin{thm}[Dobrinen-Todorcevic \cite{Dobrinen/Todorcevic10}]\label{thm.5}
Suppose $\mathcal{U}$ is a p-point on $\bN$ and that $\mathcal{V}$ is an arbitrary ultrafilter on $\bN$ such that $\mathcal{U}\ge_T \mathcal{V}$.
Then there is a continuous monotone map
$g:\mathcal{P}(\bN)\ra\mathcal{P}(\bN)$
 whose restriction to $\mathcal{U}$ is continuous
and has cofinal range in $\mathcal{V}$.
Hence,  $g\re\mathcal{U}$ is a continuous monotone cofinal map from $\mathcal{U}$ into $\mathcal{V}$ witnessing that $\mathcal{U}\ge_T\mathcal{V}$.
\end{thm}

The proof of Theorem \ref{thm.5} actually gives a type of canonization for monotone cofinal maps on p-points:
If $\mathcal{U}$ is a p-point and $f:\mathcal{U}\ra\mathcal{V}$ is a monotone cofinal map, then there is an $\tilde{X}\in\mathcal{U}$ such that the restriction of $f$ to $\mathcal{U}\re \tilde{X}$ is continuous.
For further background and results on continuous cofinal maps in relation to Tukey types of ultrafilters,  the reader is referred to \cite{Dobrinen/Todorcevic10} and \cite{Dobrinen10}.

Even though p-points have Tukey types of cardinality continuum,  in general, the Tukey type of a p-point is quite different from its Rudin-Keisler isomorphism class.
To discuss this further, the reader is reminded of the definition of the Fubini product of a collection of ultrafilters.

\begin{defn}\label{defn.Fubprod}
Let $\mathcal{U},\mathcal{V}_n$, $n<\om$, be ultrafilters.
The {\em Fubini product} of $\mathcal{U}$ and $\mathcal{V}_n$, $n<\om$,
is the ultrafilter, denoted $\lim_{n\ra\mathcal{U}}\mathcal{V}_n$,  on base set $\om\times\om$ consisting of the sets $A\sse\om\times\om$ such that
\begin{equation}
\{n\in\om:\{j\in\om:(n,j)\in A\}\in\mathcal{V}_n\}\in\mathcal{U}.
\end{equation}
That is, for $\mathcal{U}$-many $n\in\om$, the section $(A)_n$ is in $\mathcal{V}_n$.
If all $\mathcal{V}_n=\mathcal{U}$, then we let $\mathcal{U}\cdot\mathcal{U}$ denote $\lim_{n\ra\mathcal{U}}\mathcal{U}$.
\end{defn}

It is well-known that the Fubini product of two or more  p-points is not a p-point, hence for any p-point, $\mathcal{U}\cdot\mathcal{U}>_{RK}\mathcal{U}$.
In Corollary 37 of \cite{Dobrinen/Todorcevic10}, it was shown that every Ramsey ultrafilter $\mathcal{V}$ has Tukey type equal to the Tukey type of $\mathcal{V}\cdot\mathcal{V}$, and moreover that this is the case for any rapid p-point.
Further, in  Theorem 25 of \cite{Raghavan/Todorcevic11}, Raghavan and the second author showed  that, assuming CH, there are p-points $\mathcal{U}\equiv_T\mathcal{V}$ such that $\mathcal{V}<_{RK}\mathcal{U}$.
By these results, we see that, although the Tukey type of any p-point has size continuum, it contains many Rudin-Keisler inequivalent  ultrafilters within it.
One may reasonably ask what the structure of the isomorphism classes within the Tukey type of a p-point is.

For Ramsey ultrafilters, the picture has been made clear.

\begin{thm}[Todorcevic,  Theorem 24, \cite{Raghavan/Todorcevic11}]\label{thm.tod}
If $\mathcal{U}$ is a Ramsey  ultrafilter and $\mathcal{V}\le_T\mathcal{U}$,
then $\mathcal{V}$ is isomorphic to a
countable iterated Fubini product of $\mathcal{U}$.
\end{thm}

As discussed in the Section \ref{sec.overview}, the  proof of Theorem \ref{thm.tod} uses the Pudlak-\Rodl\ Theorem \ref{thm.PR}  which we  review below.

Given Theorem \ref{thm.tod}, one may reasonably ask whether a similar situation holds for ultrafilters which are not Ramsey but are low in the Rudin-Keisler hierarchy.
The most natural place to start is with an ultrafilter which is weakly Ramsey but not Ramsey.
Laflamme forced such an ultrafilter which has extra partition properties which allow for complete combinatorics.
Recall from \cite{Laflamme89} that an ultrafilter $\mathcal{U}$ is said to satisfy the {\em $(n,k)$ Ramsey partition property} if for all functions $f:[\om]^k\ra n^{k-1}+1$, and all partitions $\lgl A_m:m\in\om\rgl$ of  $\om$ with each $A_m\not\in\mathcal{U}$,
there is a set $X\in\mathcal{U}$ such that $|X\cap A_m|<\om$ for each $m<\om$, and $|f''[A_m\cap X]^2|\le n^{k-1}$ for each $m<\om$.

\begin{thm}[Laflamme]\label{thm.Laflammethms}
One can force an ultrafilter $\mathcal{U}_1$, by a $\sigma$-complete forcing $\bP_1$, with the following properties.
\begin{enumerate}
\item
\rm [Proposition 1.6, \cite{Laflamme89}] \it
$\mathcal{U}_1$ satisfies $(1,k)$ Ramsey partition property for all $k\ge 1$, hence $\mathcal{U}_1$ is weakly Ramsey.
\item
\rm [Proposition 1.7, \cite{Laflamme89}]  \it
$\mathcal{U}_1$ is not Ramsey.
\item
\rm [Theorem 1.15, \cite{Laflamme89}] \it
$\mathcal{U}_1$ has complete combinatorics:
Let $\kappa$ be Mahlo and $G$ be Levy$(\kappa)$-generic over $V$.  If $\mathcal{U}\in V[G]$ is a rapid ultrafilter satisfying $\RP(k)$ for all $k$ but is not Ramsey, then $\mathcal{U}$ is $\bP_1$-generic over $\HOD(\bR)^{V[G]}$.
\end{enumerate}
\end{thm}

The following theorem of Blass implies that there is only one isomorphism class Rudin-Keisler below $\mathcal{U}_1$.

\begin{thm}[Blass, Theorem 5 \cite{Blass74}]\label{thm.Blass}
Every weakly Ramsey ultrafilter has up to isomorphism only one nonprincipal Rudin-Keisler predecessor, which is a Ramsey ultrafilter.
\end{thm}

In Theorem \ref{thm.TukeyU_1} of Section \ref{sec.R1Tukey}, we extend Theorem \ref{thm.tod}.
The ultrafilter associated with $\mathcal{R}_1$ is isomorphic to $\mathcal{U}_1$, so we use the same notation to denote it.
The projection of $\mathcal{U}_1$ via a particular finite-to-one mapping produces a Ramsey ultrafilter $\mathcal{U}_0$.
In addition, there are ultrafilters which we denote $\mathcal{Y}_n$, $n\ge 2$, which are rapid p-points and are Tukey equivalent to $\mathcal{U}_1$, but are not isomorphic to $\mathcal{U}_1$.
We show in Theorem \ref{thm.TukeyU_1} that this collection of ultrafilters $\{\mathcal{U}_0,\mathcal{U}_1\}\cup\{\mathcal{Y}_n: 2\le n<\om\}$
generates, up to isomorphism, via iterated Fubini products all ultrafilters which are Tukey reducible to $\mathcal{U}_1$.
Our proof involves an application of Theorem \ref{thm.PRR(1)}, which recovers the Pudlak-\Rodl\ Theorem as a corollary.

At this point, we provide the  context for Theorem \ref{thm.PRR(1)}.
We remind the reader that $[M]^k$ denotes the collection of all subsets of the given set $M$ with cardinality $k$.
Recall the following well-known theorem of Ramsey.

\begin{thm}[Ramsey \cite{Ramsey29}]\label{thm.ramsey}
For every positive integer $k$ and every finite coloring of the family $[\bN]^k$,
there is an infinite subset $M$ of $\bN$ such that
the set $[M]^k$ of all $k$-element subsets of $M$ is monochromatic.
\end{thm}

When one is interested in  equivalence relations on $[\bN]^k$,
the canonical equivalence relations are determined by subsets $I\sse\{0,\dots,k-1\}$ as follows:
\begin{equation}
\{x_0,\dots,x_{k-1}\} \E_I \{y_0,\dots,y_{k-1}\}\mathrm{\ iff\ } (\forall i\in I)\ x_i=y_i,
\end{equation}
where the $k$-element sets $\{x_0,\dots,x_{k-1}\}$ and $\{y_0,\dots,y_{k-1}\}$ are taken to be in increasing order.

\begin{thm}[\Erdos-Rado \cite{Erdos/Rado50}]\label{thm.ER}
For every $k\ge 1$ and every equivalence relation $\E$ on $[\bN]^k$,
there is an infinite subset $M$ of $\bN$ and an index set $I\sse\{0,\dots,k-1\}$ such that $\E\re [M]^k=\E_I\re[M]^k$.
\end{thm}

 Theorem \ref{thm.ER} is a strengthening of Theorem \ref{thm.ramsey} as it allows  the coloring of $[\bN]^k$ to take on infinitely many colors:
To any equivalence relation $\E$ on $[\bN]^k$, there is  a function $f:[\bN]^k\ra\bN$ such that for all $a,b\in[\bN]^k$,
$a\E b$ iff $f(a)=f(b)$.
Conversely, each function $f:[\bN]^k\ra\bN$ partitions $[\bN]^k$ into equivalence classes via the relation $\E$ defined by $a \E b$ iff $f(a)=f(b)$.

For each $k<\om$, the set $[\bN]^k$ is an example of the more general notions of fronts and barriers.

\begin{defn}[\cite{TodorcevicBK10}]\label{defn.barrier}
Let $\mathcal{F}\sse[\bN]^{<\om}$ and $M\in[\bN]^{\om}$.
$\mathcal{F}$ is a {\em front} on $M$ if
\begin{enumerate}
\item
For each $X\in[M]^{\om}$,
there is an $a\in\mathcal{F}$ for which $a\sqsubset X$; and
\item
For all $a,b\in\mathcal{F}$ such that $a\ne b$, we have $a\not\sqsubseteq b$.
\end{enumerate}
$\mathcal{F}$ is a {\em barrier} on $M$ if (1) and (2$'$) hold,
where
\begin{enumerate}
\item[(2$'$)]
For all $a,b\in\mathcal{F}$ such that $a\ne b$, we have $a\not\sse b$.
\end{enumerate}
\end{defn}

Thus, every barrier is a front.  Moreover, by a theorem of  Galvin in \cite{Galvin68}, for every front $\mathcal{F}$, there is an infinite $M\sse\bN$ for which $\mathcal{F}|M$ is a barrier.
The Pudlak-\Rodl\ Theorem  extends the \Erdos-Rado Theorem to general barriers.
If $\mathcal{F}$ is a front, a mapping $\vp:\mathcal{F}\ra\bN$ is called {\em irreducible} if it is (a) {\em inner}, meaning that $\vp(a)\sse a$ for all $a\in\mathcal{F}$, and (b) {\em Nash-Williams},
meaning that for each $a,b\in\mathcal{F}$, $\vp(a)\not\sqsubset \vp(b)$.

\begin{thm}[Pudlak-\Rodl, \cite{Pudlak/Rodl82}]\label{thm.PR}
For every barrier $\mathcal{F}$ on $\bN$ and every equivalence relation $\E$ on $\mathcal{F}$, there is an infinite $M\sse\bN$ such that the restriction of $\E$ to $\mathcal{F}|M$ is represented by an irreducible mapping defined on $\mathcal{F}|M$.
\end{thm}

Our Theorem \ref{thm.PRR(1)} generalizes the Pudlak-\Rodl\ Theorem to general barriers on the topological Ramsey space $\mathcal{R}_1$.
As a corollary, we obtain Theorem \ref{thm.original}, a generalization of the \Erdos-Rado Theorem to barriers on $\mathcal{R}_1$ which are the analogues of $[\bN]^n$.

The paper is organized as follows.
The space $\mathcal{R}_1$ is introduced in Section \ref{sec.tRs} and is proved to be a topological Ramsey space.
Section \ref{sec.canonizationsR_1} contains the Ramsey-classification  Theorems \ref{thm.original} and \ref{thm.PRR(1)} for barriers on $\mathcal{R}_1$.
Then Theorem \ref{thm.PRR(1)} is applied in Section \ref{sec.R1Tukey} to classify the Rudin-Keisler types within the Tukey types of ultrafilters Tukey reducible to $\mathcal{U}_1$.

%******************************************************************************************************************************************************************

\section{The topological Ramsey space $\mathcal{R}_1$}\label{sec.tRs}

Recall that the Ellentuck space
consists of  $[\bN]^{\om}$, the collection of all infinite subsets of $\bN$ enumerated in strictly increasing order, along with the topology given by the basic open sets
$[a,B]:=\{A\in[\bN]^{\om}:a\sqsubseteq A$ and $A\sse B\}$,
where $a$ is a finite subset of $\bN$ and $B\in[\bN]^{\om}$.
This topology is a refinement of the usual metric topology on $[\bN]^{\om}$ produced by the clopen sets $[a,\bN]$, for $a$ a finite subset of $\bN$.
The Ellentuck space is the fundamental example of the more general notion of a topological Ramsey space.

For the convenience of the reader, we include the following definitions and theorems from Chapter 5, Section 1 \cite{TodorcevicBK10}.
The  axioms \bf A.1 \rm - \bf A.4 \rm
are defined for triples
$(\mathcal{R},\le,r)$
of objects with the following properties.
$\mathcal{R}$ is a nonempty set,
$\le$ is a quasi-ordering on $\mathcal{R}$,
 and $r:\mathcal{R}\times\om\ra\mathcal{AR}$ is a mapping giving us the sequence $(r_n(\cdot)=r(\cdot,n))$ of approximation mappings, where
$\mathcal{AR}$ is  the collection of all finite approximations to members of $\mathcal{R}$.
For $a\in\mathcal{AR}$ and $A,B\in\mathcal{R}$,
\begin{equation}
[a,B]=\{A\in\mathcal{R}:A\le B\mathrm{\ and\ }(\exists n)\ r_n(A)=a\}.
\end{equation}

For $a\in\mathcal{AR}$, let $|a|$ denote the length of the sequence $a$.  Thus, $|a|$ equals the integer $k$ for which $a=r_k(a)$.
For $a,b\in\mathcal{AR}$, $a\sqsubseteq b$ if and only if $a=r_m(b)$ for some $m\le |b|$.
$a\sqsubset b$ if and only if $a=r_m(b)$ for some $m<|b|$.
For each $n<\om$, $\mathcal{AR}_n=\{r_n(A):A\in\mathcal{R}\}$.
If $n>|a|$, then  $r_n[a,A]$ is the collection of all $b\in\mathcal{AR}_n$ such that $a\sqsubset b$ and $b\le_{\mathrm{fin}} A$.
\vskip.1in

\begin{enumerate}
\item[\bf A.1]\rm
\begin{enumerate}
\item
$r_0(A)=\emptyset$ for all $A\in\mathcal{R}$.\vskip.05in
\item
$A\ne B$ implies $r_n(A)\ne r_n(B)$ for some $n$.\vskip.05in
\item
$r_n(A)=r_m(B)$ implies $n=m$ and $r_k(A)=r_k(B)$ for all $k<n$.\vskip.1in
\end{enumerate}
\item[\bf A.2]\rm
There is a quasi-ordering $\le_{\mathrm{fin}}$ on $\mathcal{AR}$ such that\vskip.05in
\begin{enumerate}
\item
$\{a\in\mathcal{AR}:a\le_{\mathrm{fin}} b\}$ is finite for all $b\in\mathcal{AR}$,\vskip.05in
\item
$A\le B$ iff $(\forall n)(\exists m)\ r_n(A)\le_{\mathrm{fin}} r_m(B)$,\vskip.05in
\item
$\forall a,b,c\in\mathcal{AR}[a\sqsubset b\wedge b\le_{\mathrm{fin}} c\ra\exists d\sqsubset c\ a\le_{\mathrm{fin}} d]$.\vskip.1in
\end{enumerate}
\end{enumerate}

$\depth_B(a)$ is the least $n$, if it exists, such that $a\le_{\mathrm{fin}}r_n(B)$.
If such an $n$ does not exist, then we write $\depth_B(a)=\infty$.
If $\depth_B(a)=n<\infty$, then $[\depth_B(a),B]$ denotes $[r_n(B),B]$.

\begin{enumerate}
\item[\bf A.3] \rm
\begin{enumerate}
\item
If $\depth_B(a)<\infty$ then $[a,A]\ne\emptyset$ for all $A\in[\depth_B(a),B]$.\vskip.05in
\item
$A\le B$ and $[a,A]\ne\emptyset$ imply that there is $A'\in[\depth_B(a),B]$ such that $\emptyset\ne[a,A']\sse[a,A]$.\vskip.1in
\end{enumerate}
\end{enumerate}

\begin{enumerate}
\item[\bf A.4]\rm
If $\depth_B(a)<\infty$ and if $\mathcal{O}\sse\mathcal{AR}_{|a|+1}$,
then there is $A\in[\depth_B(a),B]$ such that
$r_{|a|+1}[a,A]\sse\mathcal{O}$ or $r_{|a|+1}[a,A]\sse\mathcal{O}^c$.\vskip.1in
\end{enumerate}

The topology on $\mathcal{R}$ is given by the basic open sets
$[a,B]$.
This topology is called the {\em natural} or {\em Ellentuck} topology on $\mathcal{R}$;
it extends the usual metrizable topology on $\mathcal{R}$ when we consider $\mathcal{R}$ as a subspace of the Tychonoff cube $\mathcal{AR}^{\bN}$.
Given the Ellentuck topology on $\mathcal{R}$,
the notions of nowhere dense, and hence of meager are defined in the natural way.
Thus, we may say that a subset $\mathcal{X}$ of $\mathcal{R}$ has the {\em property of Baire} iff $\mathcal{X}=\mathcal{O}\cap\mathcal{M}$ for some Ellentuck open set $\mathcal{O}\sse\mathcal{R}$ and Ellentuck meager set $\mathcal{M}\sse\mathcal{R}$.

\begin{defn}[\cite{TodorcevicBK10}]\label{defn.5.2}
A subset $\mathcal{X}$ of $\mathcal{R}$ is {\em Ramsey} if for every $\emptyset\ne[a,A]$,
there is a $B\in[a,A]$ such that $[a,B]\sse\mathcal{X}$ or $[a,B]\cap\mathcal{X}=\emptyset$.
$\mathcal{X}\sse\mathcal{R}$ is {\em Ramsey null} if for every $\emptyset\ne [a,A]$, there is a $B\in[a,A]$ such that $[a,B]\cap\mathcal{X}=\emptyset$.

A triple $(\mathcal{R},\le,r)$ is a {\em topological Ramsey space} if every property of Baire subset of $\mathcal{R}$ is Ramsey and if every meager subset of $\mathcal{R}$ is Ramsey null.
\end{defn}

We shall need the following result which can be found as Theorem
5.4 in \cite{TodorcevicBK10}.

\begin{thm}[Abstract Ellentuck Theorem]\label{thm.AET}\rm \it
If $(\mathcal{R},\le,r)$ is closed (as a subspace of $\mathcal{AR}^{\bN}$) and satisfies axioms {\bf A.1}, {\bf A.2}, {\bf A.3}, and {\bf A.4},
then every property of Baire subset of $\mathcal{R}$ is Ramsey,
and every meager subset is Ramsey null;
in other words,
the triple $(\mathcal{R},\le,r)$ forms a topological Ramsey space.
\end{thm}

Extensions  of the Silver and Galvin-Prikry Theorems to  topological Ramsey spaces have been proved in \cite{TodorcevicBK10}.
In particular,
every topological Ramsey space has the property that every Souslin-measurable set is Ramsey.
See Chapter 5 of \cite{TodorcevicBK10} for further information.

Certain types of subsets of the collection of approximations $\mathcal{AR}$ of a given topological Ramsey space have the Ramsey property.

\begin{defn}[\cite{TodorcevicBK10}]\label{defn.5.16}
A family $\mathcal{F}\sse\mathcal{AR}$ of finite approximations is
\begin{enumerate}
\item
{\em Nash-Williams} if $a\not\sqsubseteq b$ for all $a\ne b\in\mathcal{F}$;
\item
{\em Sperner} if $a\not\le_{\mathrm{fin}} b$ for all $a\ne b\in\mathcal{F}$;
\item
{\em Ramsey} if for every partition $\mathcal{F}=\mathcal{F}_0\cup\mathcal{F}_1$ and every $X\in\mathcal{R}$,
there are $Y\le X$ and $i\in\{0,1\}$ such that $\mathcal{F}_i|Y=\emptyset$.
\end{enumerate}
\end{defn}

The next theorem appears as  Theorem 5.17  in \cite{TodorcevicBK10}.

\begin{thm}[Abstract Nash-Williams Theorem]\label{thm.ANW}
Suppose $(\mathcal{R},\le,r)$ is a closed triple that satisfies {\bf A.1} - {\bf A.4}. Then
every Nash-Williams family of finite approximations is Ramsey.
\end{thm}

\begin{defn}\label{def.frontR1}
Suppose $(\mathcal{R},\le,r)$ is a closed triple that satisfies {\bf A.1} - {\bf A.4}.
Let $X\in\mathcal{R}$.
A family $\mathcal{F}\sse\mathcal{AR}$ is a {\em front} on $[0,X]$ if
\begin{enumerate}
\item
For each $Y\in[0,X]$, there is an $a\in \mathcal{F}$ such that $a\sqsubset Y$; and
\item
$\mathcal{F}$ is Nash-Williams.
\end{enumerate}
$\mathcal{F}$ is a {\em barrier}  if (1) and ($2'$) hold,
where
\begin{enumerate}
\item[(2$'$)]
$\mathcal{F}$ is Sperner.
\end{enumerate}
\end{defn}

\begin{rem}
Any front on a topological Ramsey space is Nash-Williams; hence is Ramsey, by Theorem \ref{thm.ANW}.
\end{rem}

Now we introduce the topological Ramsey space $(\mathcal{R}_1,\le_1,r)$.
This space was inspired by
 Laflamme's forcing $\bP_1$ which adds an ultrafilter $\mathcal{U}_1$ which is
not Ramsey, but is weakly Ramsey in a strong sense.
$\mathcal{R}_1$ forms a dense subset of $\bP_1$.
Much more will be said about this in Section \ref{sec.R1Tukey}.

\begin{defn}[$(\mathcal{R}_1,\le_1,r)$]\label{defn.R_1}
Let $\mathbb{T}$ denote the following infinite tree of height $2$.
\begin{equation}
\mathbb{T}=\{\lgl\rgl\}\cup\{\lgl n\rgl:n<\om\}\cup\bigcup_{n<\om}\{\lgl n,i\rgl:i\le n\}.
\end{equation}
$\mathbb{T}$ is to be thought of as an infinite sequence of finite trees of height $2$,
where the {\em $n$-th subtree} of $\mathbb{T}$ is
\begin{equation}
\mathbb{T}(n)=\{\lgl\rgl,\lgl n\rgl,\lgl n,i\rgl:i\le n\}.
\end{equation}
The members $X$ of $\mathcal{R}_1$ are infinite subtrees of $\mathbb{T}$ which have the same structure as $\mathbb{T}$.
That is, a tree $X\sse\mathbb{T}$ is in $\mathcal{R}_1$ if and only if
 there is a strictly increasing sequence $(k_n)_{n<\om}$ such that
\begin{enumerate}
\item
  $X\cap \mathbb{T}(k_n)\cong \mathbb{T}(n)$ for each $n<\om$; and
\item
   whenever $X\cap \mathbb{T}(j)\ne\emptyset$, then $j=k_n$ for some $n<\om$.
\end{enumerate}
We let $X(n)$ denote $X\cap \mathbb{T}(k_n)$.
We shall call $X(n)$ the {\em $n$-th tree} of $X$.
For $n<\om$, $r_n(X)$ denotes $\bigcup_{i<n}X(i)$.
$\mathcal{AR}_n=\{r_n(X):X\in\mathcal{R}_1\}$,
 and
$\mathcal{AR}=\bigcup_{n<\om}\mathcal{AR}_n$.

For $X,Y\in\mathcal{R}_1$,
define $Y\le_1 X$ if and only if
there is a strictly increasing sequence $(k_n)_{n<\om}$
such that for each $n$, $Y(n)$ is a subtree of $X(k_n)$.
Let $a,b\in\mathcal{AR}$ and $A,B\in\mathcal{R}_1$.
The quasi-ordering $\le_{\mathrm{fin}}$ on $\mathcal{AR}$ is defined as follows:
$b\le_{\mathrm{fin}} a$ if and only if there are $n\le m$ and a strictly increasing sequence $(k_i)_{i<n}$  with $k_{n-1}<m$ such that
$a\in\mathcal{AR}_m$,
 $b\in\mathcal{AR}_n$, and
for each $i< n$,  $b(i)$ is a subtree of $a(k_i)$.
We write $a\le_{\mathrm{fin}} B$ if and only if there is an $n$ such that $a\le_{\mathrm{fin}} r_n(B)$.
The basic open sets are given by
$[a,B]=\{X\in\mathcal{R}_1:a\sqsubseteq X$ and $X\le_1 B\}$.
\end{defn}

\begin{rem}
Because of the structure of $\bT$ and the definition of $\mathcal{R}_1$,
it turns out that for any two $X,Y\in\mathcal{R}_1$,
 $Y\le_1 X$ if and only if $Y\sse X$.
 Likewise, for any $a,b\in\mathcal{AR}$,
$a\le_{\mathrm{fin}}b$ if and only if $a\sse b$. 
\end{rem}

We now present some notation which will be quite useful in the next section.
$A/b$ denotes $A\setminus r_n(A)$, where $n$ is least such that $\depth_{\bT}(r_n(A))\ge\depth_{\bT}(b)$.
$\mathcal{R}_1(k)=\{X(k):X\in\mathcal{R}_1\}$;
 $\mathcal{R}_1(k)|A=\{X(k):X\in\mathcal{R}_1$ and $X(k)\sse A\}$;
 and
$\mathcal{R}_1(k)|A/b=\{X(k):X\in\mathcal{R}_1$, $X(k)\sse A/b\}$.

We now arrive at the main fact about $\mathcal{R}_1$ of this section.

\begin{thm}
$(\mathcal{R}_1,\le,r)$ is a topological Ramsey space.
\end{thm}

\begin{proof}
By the Abstract Ellentuck Theorem,
it suffices to show that $(\mathcal{R}_1,\le_1,r)$  is a closed subspace of the Tychonov power $\mathcal{AR}^{\bN}$ of $\mathcal{AR}$ with its discrete topology, and that $(\mathcal{R}_1,\le_1,r)$ satisfies axioms {\bf A.1} - {\bf A.4}.

$\mathcal{R}_1$ is identified with the subspace of $\mathcal{AR}^{\bN}$  consisting of all sequences $\lgl a_n:n<\om\rgl$ such that there is an $A\in\mathcal{R}_1$ such that for each $n<\om$,
$a_n=r_n(A)$.
That $\mathcal{R}_1$ is a closed subspace of
$\mathcal{AR}^{\bN}$ follows from the fact that
given any sequence $\lgl a_n:n<\om\rgl$ such that each $a_n\in\mathcal{AR}_n$ and $r_n(a_k)=a_n$ for each $k\ge n$, the union  $A=\bigcup_{n<\om}a_n$ is a member of $\mathcal{R}_1$.

\bf A.1.  \rm
(1) By definition, $r_0(A)=\emptyset$ for all $A\in\mathcal{R}_1$.
(2) $A\ne B$ implies that for some $n$, $r_n(A)\ne r_n(B)$.
(3) If $r_n(A)=r_m(B)$, then it must be the case that $n=m$ and $r_k(A)=r_k(B)$ for all $k<n$.

\bf A.2. \rm
(1) For each $b\in\mathcal{AR}$, there is a unique $n$ such that $b\in\mathcal{AR}_n$.
So,
\begin{equation}
\{a\in\mathcal{AR}:a\le_{\mathrm{fin}} b\}=
\bigcup_{k\le n}\{a\in\mathcal{AR}_k: \forall i\le k\ \exists m_i\le n\, (a(i)\sse b(m_i))\}.
\end{equation}
This set is finite.
(2) $A\le_1 B$ if and only if for each $n$ there is an $m$ such that $r_n(A)\le_{\mathrm{fin}} r_m(B)$.
This is clear from the definition.
(3)
For each $a,b\in\mathcal{AR}$, if $a\sqsubseteq b$ and $b\le_{\mathrm{fin}} c$,
then in fact $a\le_{\mathrm{fin}} c$.

\bf A.3. \rm
(1)
If $\depth_B(a)=n<\infty$,
then  $a\le_{\mathrm{fin}} r_n(B)$.
If  $A\in[\depth_B(a),B]$,
then $r_n(A)=r_n(B)$ and for each $k>n$,
there is an $m_k$ such that $A(k)\sse B(m_k)$.
Letting $l$ be such that $a\in\mathcal{AR}_l$,
for each $i\ge 1$,
let $w(l+i)$ be any subtree of $A(n+i)$ isomorphic to  $\mathbb{T}(l+i)$.
Let $A'=a\cup\bigcup\{w(l+i):i\ge 1\}$.
Then $A'\in[a,A]$, so $[a,A]\ne\emptyset$.

(2) Suppose $A\le_1 B$ and $[a,A]\ne\emptyset$.
Then $\depth_B(a)<\infty$ since $A\le_1 B$.
Let $n=\depth_B(a)$ and $k=\depth_A(a)$.
Note that $k\le n$ and for each $j\ge k$, $A(j)\sse B(l)$ for some $l\ge n$.
Let $A'=r_n(B)\cup\bigcup\{A(n+i):i<\om\}$.
Then $A'\in[\depth_B(a),B]$ and
$\emptyset\ne[a,A']\sse[a,A]$.
%(For letting $m$ be such that $a\in\mathcal{AR}_m$,
%taking a subtree $w(m+i)\sse A'(n+i)$ isomorphic to $\mathbb{T}(m+i)$ for each $i\ge 1$, and letting $C=a\cup\bigcup\{w(m+i):i\ge 1\}$,
%we have that $C\in[a,A']$.
%Moreover, for any $C\in[a,A']$,
%we have that every $C(i)\sse A(k_i)$ for some unique $k_i$,
% so $[a,A']\sse[a,A]$.)

\bf A.4. \rm
Suppose that $\depth_B(a)=n<\infty$ and $\mathcal{O}\sse\mathcal{AR}_{|a|+1}$.
Let $k=|a|$.
Recall that
 $r_{k+1}[a,B]$ is defined to be the collection of $c\in\mathcal{AR}_{k+1}$ such that $r_k(c)=r_k(a)$ and $c(k)$ is a subtree of $B(m)$ for some $m\ge n$.
So we may think of $\mathcal{O}$ as a 2-coloring on the collection of  subtrees $u\sse B(m)$ isomorphic to $\mathbb{T}(k)$ for some $m\ge n$.

Say a set $u\in\mathcal{R}_1(k)|B/r_n(B)$
has color $0$ if $a\cup u$ is in $\mathcal{O}$ and has color $1$ if $a\cup u$ is in $\mathcal{O}^c$.
Identifying each tree isomorphic to $\bT(m)$ with its leaves, the Finite Ramsey Theorem may be applied.
By the Finite Ramsey Theorem,  taking $N_0$ large enough, there is a subtree $w(n)\sse
B(N_0)$ isomorphic to $\mathbb{T}(n)$ such that the collection of all subtrees of $w(n)$ which are isomorphic to $\mathbb{T}(k)$ is monochromatic.
Take $N_1>N_0$ large enough that there is a subtree $w(n+1)\sse
B(N_1)$ isomorphic to  $\mathbb{T}(n+1)$ such that
the collection of all subtrees of $w(n+1)$ which are isomorphic to $\mathbb{T}(k)$  is monochromatic.
In general, given $N_i$ and $w(n+i)$,
take $N_{i+1}>N_i$ large enough that there is a subtree $w(n+i+1)\sse B(N_{i+1})$ isomorphic to  $\bT(n+i+1)$ such that
the collection of all subtrees of $w(n+i+1)$ which are isomorphic to $\mathbb{T}(k)$
is monochromatic.
Now the colors on the subtrees of  $w(n+i)$
may be different for different $i$, so take a subsequence $(m_l)_{l<\om}$ of $(n+i)_{i<\om}$ such that
all the subtrees of
 $w(m_l)$
 isomorphic to $\mathbb{T}(k)$
 have the same color for all $l<\om$.
Then thin down, by taking any subtree $u(n+l)\sse w(m_l)$ isomorphic to  $\mathbb{T}(n+l)$, for each $l<\om$.
Finally, let $A=r_n(B)\cup\bigcup\{u(n+l):l<\om\}$.
Then $A\in[\depth_B(a),B]$ and  either $r_{k+1}[a,A]\sse\mathcal{O}$,
 or else
$r_{k+1}[a,A]\sse\mathcal{O}^c$.
\end{proof}

\begin{rem}\label{rem.Cor5.19}
Since for $\mathcal{R}_1$, the quasi-ordering $\le_{\mathrm{fin}}$ on $\mathcal{AR}$ is actually a partial ordering, it follows from Corollary 5.19 in \cite{TodorcevicBK10} that for any front $\mathcal{F}$ on $[0,X]$,  $X\in\mathcal{R}_1$, there is a $Y\le_1 X$ such that $\mathcal{F}|Y$ is a barrier.
\end{rem}

%******************************************************************************************************************************************************************************************************************************

\section{Canonization theorems for $\mathcal{R}_1$}\label{sec.canonizationsR_1}

This section contains the canonization theorems  for equivalence relations on fronts on the topological Ramsey space $\mathcal{R}_1$.
Theorem \ref{thm.original}  generalizes the \Erdos-Rado Theorem  for barriers of the Ellentuck space the form $[\bN]^n$ to barriers of $\mathcal{R}_1$ of the form $\mathcal{AR}_n$ for $n<\om$.
Theorem \ref{thm.PRR(1)} is the main theorem of this section, which provides canonical forms for equivalence relations on general fronts on $\mathcal{R}_1$.
This yields the Pudlak-\Rodl\ Theorem for equivalence relations for barriers on the Ellentuck space.

Recall Definition \ref{def.frontR1} of front and barrier.
Given a front $\mathcal{F}$ on some $[\emptyset, A]$ and an $X\le_1 A$, recall $\mathcal{F}|X$ denotes the collection of all $t\in\mathcal{F}$ such that $t\le_{\mathrm{fin}} X$.
Note that $\mathcal{F}|X$ forms a front on $[\emptyset,X]$.
More generally, if $\mathcal{H}$ is any subset of $\mathcal{AR}$ and $X\in\mathcal{R}_1$, we write $\mathcal{H}|X$ to denote the collection of all $t\in\mathcal{H}$ such that $t\le_{\mathrm{fin}} X$.
Henceforth, we drop the subscript on $\le_1$ and just write $\le$.

We begin by setting  up  notation regarding equivalence relations.

\begin{defn}\label{defn.ET}
For each $n<\om$, let $\tilde{T}(n)$ denote the tree  $\{\lgl\rgl,\lgl 0\rgl, \lgl 0,i\rgl:i\le n\}$.
Let $T_{\lgl\rgl}=\{\lgl\rgl\}$ and
let $T_{\lgl 0\rgl}=\{\lgl\rgl,\lgl 0\rgl\}$.
For $\emptyset\ne I\sse n+1$,
let $T_I=\{\lgl\rgl,\lgl 0\rgl,\lgl 0,i\rgl:i\in I\}$.
Let $\mathcal{T}(n)$ denote  the collection of all (downwards closed) subtrees of $\tilde{T}(n)$ of any height.
Thus, $\mathcal{T}(n)$ consists of the trees $T_{\lgl\rgl}$, $T_{\lgl 0\rgl}$, and $T_I$ where $I$ is a nonempty subset of $n+1$.

Given a tree $T\in\mathcal{T}(n)$ and $X\in\mathcal{R}_1$,
let $\pi_{T}(X(n))$ denote
the $T$-projection of $X(n)$;
that is,
the subtree of $X(n)$ consisting of the nodes in those positions occurring in $T$.
Thus, if $X(n)=\{\lgl\rgl,\lgl k\rgl, \lgl k,l\rgl:l\in L\}$, where $L=\{l_0,\dots,l_n\}$,
then,
(i) $\pi_{T_{\lgl\rgl}}(X(n))=\{\lgl\rgl\}$,
(ii) $\pi_{T_{\lgl 0\rgl}}(X(n))=\{\lgl\rgl,\lgl k\rgl\}$,
and
(iii) for $\emptyset \ne I=\{i_0,\dots,i_m\}\sse n+1$, $\pi_{T_I}(X(n))=\{\lgl\rgl,\lgl k\rgl, \lgl k, l_{i_0}\rgl,\dots,\lgl k,l_{i_m}\rgl\}$.

Each $T\in\mathcal{T}(n)$ induces an equivalence relation $\E_{T}$ on $\mathcal{R}_1(n)$  in the following way:
\begin{equation}
X(n)\E_T Y(n) \Leftrightarrow
\pi_{T}(X(n))=\pi_{T}(Y(n)).
\end{equation}
Let $\mathcal{E}(n)$ denote the collection of equivalence relations $\E_T$, for $T\in\mathcal{T}(n)$.
\end{defn}

\begin{defn}\label{defn.canonicalARn}
Let $1\le n <\om$ be fixed.
An equivalence relation $\R$  on $\mathcal{AR}_n$ is {\em canonical} if and only if
there are trees $T(0)\in\mathcal{T}(0),\ \dots,\  T(n-1)\in\mathcal{T}(n-1)$ such that for all $a,b\in\mathcal{AR}_n$,
\begin{equation}
a\R b \ \Leftrightarrow \ \forall i<n\, (\pi_{T(i)}(a(i))=\pi_{T(i)}(b(i))).
\end{equation}
\end{defn}

We now are ready to state our first canonization theorem.
We remark that for each $n<\om$, $\mathcal{AR}_n$ is a barrier.

\begin{thm}\label{thm.original}
Let $1\le n<\om$.
Given any $A\in\mathcal{R}_1$ and any equivalence relation $\R$ on  $\mathcal{AR}_n|A$,
there is a $D\le A$ such that $\R$ is canonical on $\mathcal{AR}_n|D$.
\end{thm}

\begin{rem}\label{rem.numbers}
For each $1\le n<\om$, there are $\Pi_{i=1}^n(2^i+1)$ canonical equivalence relations on $\mathcal{AR}_n$.
Each  $i$-th component of the product is exactly the number of  \Erdos-Rado canonical equivalence relations on $[\bN]^i$ plus one.
\end{rem}

Though  Theorem \ref{thm.original} can be proved directly, in order to avoid unnecessary length in this paper, we shall prove it at the end of this section by a short  application of  Theorem \ref{thm.PRR(1)}.
We begin  with some general facts and lemmas which provide tools for the proof of the main theorem of this section.  In what follows, $X/(s,t)$ denotes $X/s\cap X/t$.

\begin{fact}\label{fact.B(n)}
Suppose $n<\om$, $a\in\mathcal{AR}_n$, and $B\in\mathcal{R}_1$ such that $B(n)\sse\mathbb{T}(k')$ and $a(n-1)\sse\mathbb{T}(k)$ for some $k<k'$.
Then $a\cup (B/r_n(B))$ is a member of $\mathcal{R}_1$.
\end{fact}

\begin{lem}\label{lem.1}
(1)
Suppose $P(\cdot,\cdot)$ is a property  such that for each $s\in\mathcal{AR}$ and each $X\in\mathcal{R}_1$, there is a $Z\le X$ such that $P(s,Z)$ holds.
Then for each $X\in\mathcal{R}_1$,
there is a $Y\le X$ such that for each $s\in\mathcal{AR}|Y$ and each $Z\le Y$, $P(s,Z/s)$ holds.

(2)
Suppose $P(\cdot,\cdot,\cdot)$ is a property such that for all $s,t\in\mathcal{AR}$ and each $X\in\mathcal{R}_1$, there is a $Z\le X$ such that $P(s,t,Z)$ holds.
Then for each $X\in\mathcal{R}_1$, there is a $Y\le X$ such that for all $s,t\in\mathcal{AR}|Y$ and all $Z\le Y$, $P(s,t,Z/(s,t))$ holds.
\end{lem}

\begin{proof}
The proofs are by straightforward fusion arguments.
 Let $X$ be given.  By the hypothesis, there is an $X_1\le X$ for which $P(\emptyset,X_1)$ holds.  Fix $y_1=r_1(X_1)$.
For $n\ge 1$, given $X_n$ and $y_n$, enumerate $\mathcal{AR}|y_n$ as $s_i$, $i<|\mathcal{AR}|y_n|$.
Applying the hypothesis finitely many times, we obtain an $X_{n+1}\le X_n$ such that $P(s_i,X_{n+1}/s_i)$ holds for all $i$. Let $y_{n+1}=y_n\cup X_{n+1}(n)$.
Continuing in this manner, we obtain  $Y=\bigcup_{n\ge 1}y_n$ which satisfies (1).

 Let $X$ be given.
Fix $s=r_0(X)=\emptyset$ and $t=r_1(X)$, and let
$y_1 =r_1(X)$.
By the hypothesis, there is an $X_2\le X$
such that $P(s,t,X_2)$.
Let $y_2=y_1\cup X_2(1)$.
Let $n\ge 2$  be given, and suppose  $X_n$ and $y_n$ have been constructed.
Enumerate the
 pairs of distinct elements  $s,t\in \mathcal{AR}|y_n$
as $(s_i,t_i)$, for all $i<|[\mathcal{AR}|y_n]^2|$.
By finitely many applications of the hypothesis, we obtain an $X_{n+1}\le X_{n}$ such that for each $i$,
$P(s_i,t_i,X_n)$ holds.
Let $y_{n+1}=y_n\cup X_{n+1}(n)$.
In this way we obtain $Y=\bigcup_{n\ge 1}  y_n$ which satisfies (2).
\end{proof}

Given a front $\mathcal{F}$ on $[\emptyset,A]$ for some $A\in\mathcal{R}_1$  and $f:\mathcal{F}\ra\bN$, we adhere to the following convention:
If we write $f(b)$ or $f(s\cup u)$,
it is assumed that $b,s\cup u$ are in $\mathcal{F}$.
Define
\begin{equation}
\hat{\mathcal{F}}=\{r_m(b): b\in\mathcal{F},\ m\le n<\om, \mathrm{\ where\ } b\in\mathcal{AR}_n\}.
\end{equation}
Note that $\emptyset\in\hat{\mathcal{F}}$, since $\emptyset=r_0(b)$ for any $b\in\mathcal{F}$.
For any $X\le A$,
define
\begin{equation}
\Ext(X)=\{s\setminus r_m(s):
m<\om,\
\exists n\ge m\, (s\in \mathcal{AR}_n,\mathrm{\ and\ } s\setminus r_m(s)\sse X)\}.
\end{equation}
$\Ext(X)$ is the collection of all possible legal extensions into $X$.
For any $s\in\mathcal{AR}$,  let $\Ext(X/s)$ denote the collection of those $y\in \Ext(X)$ such that
$y\sse X/s$.
For $u\in\mathcal{AR}$,
we write $v\in\Ext(u)$
to mean that
$v\in\Ext(\mathbb{T})$ and $v\sse u$.

The next notions of separating and mixing have their roots in the paper \cite{Proml/Voigt85}, where 
Pr\"{o}ml and Voigt canonized Borel mappings from $[\om]^{\om}$ into the real numbers.
We introduce  notions of separating and mixing for our context.

\begin{defn}\label{def.sepmix}
Fix $s,t\in\hat{\mathcal{F}}$ and $X\in\mathcal{R}_1$.
 $X$ {\em separates $s$ and $t$} if and only
for all  $x\in\Ext(X/s)$ and $y\in\Ext(X/t)$ such that $s\cup x$ and $t\cup y$ are in $\mathcal{F}$,
$f(s\cup x)\ne f(t\cup y)$.
 $X$ {\em mixes $s$ and $t$} if and only if there is no $Y\le X$ which separates $s$ and $t$.
$X$ {\em decides for $s$ and $t$} if and only if either $X$ separates $s$ and $t$ or else $X$ mixes $s$ and $t$.
\end{defn}

Thus, $X$ {\em mixes $s$ and $t$}
 if and only if for each $Y\le X$, there are $x,y\in\Ext(Y)$ such that $f(s\cup x)=f(t\cup y)$.
Note  that if $X$ mixes $s$ and $t$,
then for all $Y\le X$,
$Y$ mixes $s$ and $t$.
Likewise, if $X$ separates $s$ and $t$,
then for all $Y\le X$,
$Y$ separates $s$ and $t$.

The following modifications of the previous definitions will be  used in essential  ways in the proof of the main theorem of this section.

\begin{defn}\label{def.decide}
Fix $s,t\in\hat{\mathcal{F}}$ and $X\in\mathcal{R}_1$.
Let $\Ext(X/(s,t))$  denote $\Ext(X/s)\cap\Ext(X/t)$.
 $X/(s,t)$ {\em separates $s$ and $t$} if and only
for all $x,y\in\Ext(X/(s,t))$ such that  $s\cup x$ and $t\cup y$ are in $\mathcal{F}$,
$f(s\cup x)\ne f(t\cup y)$.
 $X/(s,t)$ {\em mixes $s$ and $t$} if and only if there is no $Y\le X/(s,t)$ which separates $s$ and $t$.
We say that $X/(s,t)$ {\em decides  for $s$ and $t$} if and only if either $X/(s,t)$ separates $s$ and $t$;
or else $X/(s,t)$ mixes $s$ and $t$.
Thus,
 $X/(s,t)$ {\em decides for $s$ and $t$} if and only if either for all  $x,y\in\Ext(X/(s,t))$,
$f(s\cup x)\ne f(t\cup y)$,
or else there is no $Y\le X/(s,t)$ which has this property.
\end{defn}

We point out that $X/(s,t)$ mixes $s$ and $t$
if and only if $X$ mixes $s$ and $t$.
However,
if  $X/(s,t)$ separates $s$ and $t$ it does not necessarily follow that $X$ separates $s$ and $t$.

\begin{lem}[Transitivity of Mixing]\label{lem.mixing}
For any $X\in\mathcal{R}_1$ and any $s,t,u\in\hat{\mathcal{F}}$,
if $X$ mixes $s$ and $t$ and $X$ mixes $t$ and $u$,
then $X$ mixes $s$ and $u$.
\end{lem}

\begin{proof}
Suppose to the contrary that $X$ does not mix $s$ and $u$.
Then there is a $Y\le X$ such that $Y$ separates $s$ and $u$.
Let $k=|s|$,
$l=|t|$, and $m=|u|$.
Shrinking $Y$ if necessary, we may assume that
$\depth_{\bT}(Y(1))>\max(\depth_{\bT}(s),\depth_{\bT}(t))$.
Let $Y_s=s\cup (Y\setminus r_k(Y))$ and $Y_t=t\cup(Y\setminus r_l(Y))$.
Then $Y_s$ and $Y_t$ are both members of $\mathcal{R}_1$.
Let
\begin{equation}
\mathcal{G}=\{v\in\mathcal{F}_t|Y_t:
\exists w\in\mathcal{F}_s|Y_s\, (f(v)=f(w))\}.
\end{equation}
By the Abstract Nash-Williams Theorem relativized to $\mathcal{F}_t$,
there is a $Z\in[t,Y_t]$ such that either $\mathcal{F}_t|Z\sse \mathcal{G}$ or else
$\mathcal{F}_t|Z \cap\mathcal{G}=\emptyset$.

Suppose $\mathcal{F}_t|Z\sse\mathcal{G}$.
Then for each  $v\in\mathcal{F}_t|Z$,
there is a
$w\in\mathcal{F}_s|Y_s$ such that
 $f(v)=f(w)$.
Since $Y$ separates $s$ and $u$,
for each $y\in\Ext(Z/u)$ such that $u\cup y\in\mathcal{F}$,
we have that $f(w)\ne f(u\cup y)$.
Therefore,
$f(u\cup y)\ne f(v)$.
Hence, $Z$ separates $t$ and $u$, contradicting our assumption.

Suppose $\mathcal{F}_t|Z\cap\mathcal{G}=\emptyset$.
Then for each $v\in \mathcal{F}_t|Z$,
for each $w\in\mathcal{F}_s|Y_s$,
$f(v)\ne f(w)$.
Thus, $Z$ separates $s$ and $t$,
contradicting our assumption.
Therefore, $X$ must mix $s$ and $u$.
\end{proof}

Thus, the mixing relation is an equivalence relation, since mixing is trivially reflexive and symmetric.

\begin{lem}\label{claim.A}
For each $X\in\mathcal{R}_1$,
there is a $Y\le X$ such that for each $s,t\le_{\mathrm{fin}} Y$ in $\hat{\mathcal{F}}$,
$Y/(s,t)$ decides for $s$ and $t$.
\end{lem}

\begin{proof}
For $s,t\in\mathcal{AR}$ and $Y\in\mathcal{R}_1$,
let $P(s,t,Y)$ be the following property:
If $s,t\in\hat{\mathcal{F}}$,
then $Y/(s,t)$ decides for $s$ and $t$.
We will show that for each $s,t\in\hat{\mathcal{F}}$ and each $X\in\mathcal{R}_1$,
there is a $Y\le X$ which decides for $s$ and $t$.
The claim will then follow from Lemma
\ref{lem.1} (2).

Fix $X\in\mathcal{R}_1$ and $s,t\in\hat{\mathcal{F}}$.
Let
\begin{equation}
\mathcal{X}_{s,t}=\{Y\le X:\exists v,w\in\Ext(Y)\,
(f(s\cup v)=f(t\cup w))\}.
\end{equation}
Since  $\mathcal{X}_{s,t}$ is open,
by the Abstract Ellentuck Theorem
there is a $Y\le X$ such that either
$[\emptyset, Y]\sse\mathcal{X}_{s,t}$
or else $[\emptyset,Y]\cap\mathcal{X}_{s,t}=\emptyset$.
If $[\emptyset, Y]\sse\mathcal{X}_{s,t}$, then
for each $Z\le Y$, there are $v,w\in\Ext(Z)$
 such that
$f(s\cup v)=f(t\cup w)$.
Hence, $Y$ mixes $s$ and $t$.
Suppose now that $[\emptyset,Y]\cap\mathcal{X}_{s,t}=\emptyset$.
For each $v,w\in\Ext(Y)$ such that
$s\cup v,t\cup w\in\mathcal{F}$,
 $f(s\cup v)\ne f(t\cup w)$.
Thus, $Y$ separates $s$ and $t$.
In both cases, $Y$ decides for $s$ and $t$.
\end{proof}

\begin{defn}\label{defn.irred}
Let $\mathcal{F}$ be a front on $[\emptyset, X]$ for some $X\in\mathcal{R}_1$, and let $\vp$ be a function on $\mathcal{F}$.
\begin{enumerate}
\item
$\vp$ is {\em inner} if $\vp(a)$ is a subtree of $a$, for all $a\in\mathcal{F}$.
\item
$\vp$ is {\em Nash-Williams} if $\vp(a)\not\sqsubseteq \vp(b)$, for all $a\ne b\in\mathcal{F}$.
\item
$\vp$ is {\em Sperner} if $\vp(a)\not\sse \vp(b)$ for all $a\ne b\in\mathcal{F}$
%\item
%$\vp$ is {\em irreducible} if $\vp$ is inner and Nash-Williams.
\end{enumerate}
\end{defn}

\begin{defn}\label{defn.canonical}
Let $X\in\mathcal{R}_1$, $\mathcal{F}$ be a front on $[\emptyset,X]$, and  $\R$ an equivalence relation on $\mathcal{F}$.
We say that $\R$ is {\em canonical} if and only if there is an
 inner  Sperner  function
$\vp$ on $\mathcal{F}$ such that
\begin{enumerate}
\item
for all $a,b\in\mathcal{F}$, $a\R b$ if and only if $\vp(a)=\vp(b)$; and
\item
$\vp$ is maximal among all inner Sperner functions satisfying (1).
That is,
for any other inner Sperner function $\vp'$ on $\mathcal{F}$ satisfying (1), there is a $Y\le X$ such that $\vp'(a)\sse\vp(a)$ for all $a\in\mathcal{F}|Y$.
\end{enumerate}
\end{defn}

\begin{rem}\label{rem.canonical}
The map $\vp$ constructed in the
 proof of Theorem \ref{thm.PRR(1)}  is  the only such inner Sperner map  with the additional property $(*)$ that there is a $Z\le C$ such that for each $s\in\mathcal{F}|Z$ there is a $t\in\mathcal{F}$ such that $\vp(s)=\vp(t)=s\cap t$.
This will be discussed after the proof of the following main canonization theorem.
\end{rem}

Recall that by Remark \ref{rem.Cor5.19}, for each front $\mathcal{F}$ on some $[0,A]$, there is an $A'\le A$ such that $\mathcal{F}|A'$ is a barrier.
Hence, we obtain a slightly stronger result by  proving the following main theorem for fronts.

\begin{thm}\label{thm.PRR(1)}
Suppose $A\in\mathcal{R}_1$,  $\mathcal{F}$ is  a front on $[\emptyset,A],$ and  $\R$ is an equivalence relation on $\mathcal{F}$.
Then
 there is a $C\le A$ such that $\R$ is canonical on
$\mathcal{F}| C$.
\end{thm}

\begin{proof}
Let $A\in\mathcal{R}_1$,
let $\mathcal{F}$ be a given front on $[\emptyset,A]$,
and let $\R$ be an equivalence relation on $\mathcal{F}$.
Let   $f:\mathcal{F}\ra\bN$ be any mapping which induces $\R$.
By thinning if necessary, we may assume that
 $A$ satisfies Lemma \ref{claim.A}.
Let $(\hat{\mathcal{F}}\setminus\mathcal{F})|X$ denote the collection of those $t\in\hat{\mathcal{F}}\setminus\mathcal{F}$ such that $t\le_{\mathrm{fin}} X$.

\begin{claim}\label{claim.E}
There is a $B\le A$ such that for  all $s\in(\hat{\mathcal{F}}\setminus\mathcal{F})| B$,
letting $n$ denote $|s|$,
there is an equivalence relation $\E_s\in\mathcal{E}(n)$ such that,
for all $u,v\in  \mathcal{R}_1(n)|B/s$,
$B$ mixes $s\cup u$ and $s\cup v$
if and only if $u \E_s v$.
\end{claim}

\begin{proof}
For any $X\le A$ and
$s\in\mathcal{AR}|A$,
let $P(s,X)$ denote the following statement:
``If
 $s\in\hat{\mathcal{F}}\setminus\mathcal{F}$,
then
there is an equivalence relation
$\E_s\in \mathcal{E}(|s|)$ such that for all $u,v\in
\mathcal{R}_1(|s|)|X/s$,
$X$ mixes $s\cup u$ and $s\cup v$ if and only if $u\E_s v$.''
We shall show that for each $X\le A$ and $s\in\mathcal{AR}|A$, there is a $Z\le X$ for which  $P(s,Z)$ holds.
The claim  then follows from Lemma \ref{lem.1}.

Let $X\le A$ and  $s\in\hat{\mathcal{F}}\setminus\mathcal{F}$ be given, and let $n=|s|$.
Let $\R$ denote the following equivalence relation  on
$\mathcal{R}_1(n)|A/s$:
 $u\R v$ if and only if $A$ mixes $s\cup u$ and $s\cup v$.
Let
\begin{equation}
\mathcal{X}=\{X'\le X: A\mathrm{\ mixes\ }s\cup X'(n)\mathrm{\ and\ } s\cup \pi_{T_n}(X'(n+1)) \},
\end{equation}
where $T_n$ denotes $\{\lgl\rgl,\lgl 0\rgl, \lgl 0,i\rgl:i\in n\}$.
That is, $\pi_{T_n}(X'(n+1))$ is the subtree of $X'(n+1)$ consisting of all but the rightmost branch of $X'(n+1)$.
By the Abstract Ellentuck Theorem,
there is an $X'\le X$ such that either $[\emptyset,X']\sse\mathcal{X}$, or $[\emptyset,X']\cap\mathcal{X}=\emptyset$.
Thinning again, leaving off the rightmost branch of each $X'(i)$, we obtain a $Y\le X'$ such that
 either
(i) for all $u,v\in \mathcal{R}_1(n)|Y/s$,
$u\R v$; or
(ii) for all $u,v\in \mathcal{R}_1(n)|Y/s$,
 if $u\R v$  then $\pi_{T_{\lgl 0\rgl}}(u)=\pi_{T_{\lgl 0\rgl}}(v)$.
If case (i) holds, let $Z=Y$ and $\E_s=\E_{T_{\lgl\rgl}}$.

Otherwise, case (ii) holds.
For each $I\sse n+1$,
define
\begin{align}
\mathcal{Y}_I=\{Y'\le Y :\  &
\forall u,v\in\mathcal{R}_1(n)|Y'(2n+2)\cr
&  (A
\mathrm{\ mixes\ } s\cup u\mathrm{\ and\ } s\cup v\mathrm{\ iff\ }
\pi_{T_I}(u)=\pi_{T_I}(v))\}.
\end{align}
Here, we are allowing $I$ to be empty.
Let $\mathcal{Y}'=[\emptyset,Y]\setminus\bigcup_{I\sse n+1}\mathcal{Y}_I$.
Then the $\mathcal{Y}_I$, $I\sse n+1$, along with $\mathcal{Y}'$ form an open  cover of
$[\emptyset,Y]$.
By the Abstract Ellentuck Theorem,  there is a $Z\le Y$ such that either $[\emptyset,Z]\sse\mathcal{Y}_I$ for some $I\sse n+1$, or else $[\emptyset,Z]\sse\mathcal{Y}'$.
By the
 Finite \Erdos-Rado Theorem,
it cannot be the case that $[\emptyset, Z]\sse\mathcal{Y}'$.
So there is an $I\sse n+1$ for which $[\emptyset,Z]\sse\mathcal{Y}_I$.
If $I$ is nonempty,
let $\E_s$ denote the equivalence relation $\E_{T_I}$;
if $I$ is empty, let $\E_s$ denote the equivalence relation $\E_{T_{\lgl 0 \rgl}}$ .
\end{proof}

Fix $B$ be as in Claim \ref{claim.E}.
For $s\in(\hat{\mathcal{F}}\setminus\mathcal{F})|B$ and
 $n=|s|$,
let $\E_s$ be the equivalence relation for $s$ from Claim \ref{claim.E}.
We say that $s$ is
 {\em $\E_s$-mixed by $B$}, meaning that
 for all $u,v\in
\mathcal{R}_1(n)|B/s$,
$B$ mixes $s\cup u$ and $s\cup v$ if and only if $u\E_s  v$.
Let $T_s$ denote the subtree of $\tilde{T}(n)$ such that $\E_s=\E_{T_s}$.

\begin{defn}\label{def.vp^1}
For $s\in\hat{\mathcal{F}}|B$, $n=|s|$,
and $i< n$,
define
\begin{equation}
\vp_{r_i(s)}(s(i))
=\pi_{T_{r_i(s)}}(s(i)).
\end{equation}
For $s\in\mathcal{F}|B$, define
\begin{equation}
\vp(s)=\bigcup_{i<|s|}\vp_{r_i(s)}(s(i)).
\end{equation}
\end{defn}

\begin{claim}\label{claim.G}
The following are true for all $X\le B$ and all $s,t\in\hat{\mathcal{F}}| B$.
\begin{enumerate}
\item[(A1)]
Suppose $s\not\in\mathcal{F}$ and $n=|s|$.
Then $X$ mixes $s\cup u$ and $t$ for at most one $\E_{s}$ equivalence class of $u$'s in
$\mathcal{R}_1(n)|B/s$.
\item[(A2)]
If $X/(s,t)$ separates $s$ and $t$,
then $X/(s,t)$ separates $s\cup x$ and $t\cup y$ for all  $x,y\in\Ext(X/(s,t))$ such that $s\cup x,t\cup y\in\hat{\mathcal{F}}$.
\item[(A3)]
Suppose $s\not\in\mathcal{F}$ and  $n=|s|$.
Then
$T_s=T_{\lgl\rgl}$
if and only if $X$ mixes $s$ and $s\cup u$
for all $u\in\mathcal{R}_1(n)|B/s$.
\item[(A4)]
If $s\sqsubset t$
and  $\vp(s)=\vp(t)$,
then $X$ mixes $s$ and $t$.
\end{enumerate}
\end{claim}

\begin{proof}
(A1)\ \
Suppose that there are $u,v\in\mathcal{R}_1(n)|B/s$ such that $s\cup u,s\cup v\in\hat{\mathcal{F}}$,
 $u\not\E_s\, v$,
$X$ mixes $s\cup u$ and $t$, and
$X$ mixes $s\cup v$ and $t$.
Then by transitivity of mixing,
$X$ mixes $s\cup u$ and $s\cup v$.
But this contradicts the fact that $X$ $\E_s$-mixes $s$.

(A2)\ \
Suppose that $X/(s,t)$ separates $s$ and $t$.
Let $x,y\in\Ext(X/(s,t))$ be such that $s\cup x,t\cup y\in\hat{\mathcal{F}}$.
Then for any $x',y'\in\Ext(X/(s,t))$  such that $s\cup x\cup x', t\cup y\cup y'\in\mathcal{F}$,
it must be the case that $f(s\cup x\cup x')
\ne f(t\cup y\cup y')$.

(A3)
Suppose $n=|s|$ and $T_s=T_{\lgl\rgl}$.
Suppose toward a contradiction that
then $X/(s\cup u)$ separates $s$ and $s\cup u$ for some $u\in\mathcal{R}_1(n)|X/s$.
By (A2), $X/(s\cup u)$ separates $s\cup v$ and $s\cup u\cup u'$, for all $v,u'\in\Ext(X/(s\cup u))$
such that $s\cup v,s\cup u\cup u'\in\hat{\mathcal{F}}$.
But taking $u'=\emptyset$ and $v\in\mathcal{R}_1(n)|X/(s\cup u)$,
$X/(s\cup u)$ mixes $s\cup u$ and $s\cup v$, by Claim \ref{claim.E}; contradiction.
Hence, $X/(s\cup u)$ mixes $s$ and $s\cup u$ for all $u\in\mathcal{R}_1(n)|B/s$.
Conversely, if $X$ mixes $s$ and $s\cup u$ for all $u\in\mathcal{R}_1(n)|X/s$,
then, for all $u,v\in \mathcal{R}_1(n)|X/s$, $X$ mixes $s\cup u$ and $s\cup v$,
by transitivity of mixing.
Hence, $T_s$ must be $T_{\lgl\rgl}$.

(A4)\ \
 By the definition of $\vp$,
it is clear  that
 for all $|s|\le i< |t|$,
$T_{t\re i}=T_{\lgl\rgl}$.
By induction on $|s|\le i< |t|$ using (A3) and transitivity of mixing,
it follows that $X$ mixes $s$ and $t$.
\end{proof}

\begin{claim}\label{claim.Tsametype}
If $s,t\in(\hat{\mathcal{F}}\setminus\mathcal{F})|B$ are mixed by $B/(s,t)$, then $T_s$ and $T_t$ are isomorphic.
Moreover, there is a $C\le B$ such that for all $s,t\in(\hat{\mathcal{F}}\setminus\mathcal{F})|C$, for all  $u\in\mathcal{R}_1(|s|)|C/(s,t)$ and $v\in\mathcal{R}_1(|t|)|C/(s,t)$,
 $C$ mixes $s\cup u$ and $t\cup v$
 if and only if $\vp_s(u)=\vp_t(v)$.
\end{claim}

\begin{proof}
Suppose $s,t\in (\hat{\mathcal{F}}\setminus\mathcal{F})|B$ are mixed by $B/(s,t)$, and let $X\le B$.
Let $i=|s|$ and $j=|t|$.

Suppose that $T_s=T_{\lgl\rgl}$ and $T_t\ne T_{\lgl\rgl}$.
By  (A1), $B$ mixes $s$ and $t\cup v$  for at most one $\E_t$ equivalence class of $v$'s in $\mathcal{R}_1(j)|B/t$.
Since $T_t\ne T_{\lgl\rgl}$, there is a $Y\le X/(s,t)$ such that for each $v\in\mathcal{R}_1(j)|Y$,
$Y$ separates $s$ and $t\cup v$.
Since $T_s=T_{\lgl\rgl}$,
it follows from (A4) that
for all $u\in\mathcal{R}_1(i)|Y$, $Y$ mixes $s$ and  $s\cup u$.
If there are $u\in\mathcal{R}_1(i)|Y$
and $v\in\mathcal{R}_1(j)|Y$ such that $Y$ mixes $s\cup u$ and $t\cup v$,
then $Y$ mixes $s$ and $t\cup v$, by transitivity of mixing.
This contradicts that for each $v\in\mathcal{R}_1(j)|Y$,
$Y$ separates $s$ and $t\cup v$.
Therefore, all extensions of $s$ and $t$ into $Y$ are separated.
But then $s$ and $t$ are separated, contradiction.
Hence, $T_t$ must also be $T_{\lgl\rgl}$.
By a similar argument, we conclude that $T_s=T_{\lgl\rgl}$ if and only if $T_t=T_{\lgl\rgl}$.
In this case,
 $\vp_s(u)=\vp_t(v)=\{\lgl\rgl\}$ for all
$u\in\mathcal{R}_1(i)|B$ and $v\in\mathcal{R}_1(j)|B$.

Suppose now that both  $T_s$ and $T_t$ are not $T_{\lgl \rgl}$.
Let $X\le B$, $m=\max(i,j)+1$, and $k=m^m$.
Let
\begin{align}
\mathcal{Z}_{<}&
=\{Y\le X:B\mathrm{\ separates\ } s\cup Y(i)\mathrm{\ and\ }t\cup \pi_{\tilde{T}(j)}(Y(k))\}\cr
\mathcal{Z}_{>}&
=\{Y\le X:B\mathrm{\ separates\ } s\cup \pi_{\tilde{T}(i)}(Y(k))\mathrm{\ and\ }t\cup Y(j)\}.
\end{align}
Applying the Abstract Ellentuck Theorem to the sets $\mathcal{Z}_<$ and $\mathcal{Z}_>$, we  obtain an $X'\le X$ such that, for all $u\in\mathcal{R}_1(i)|X'$ and $v\in\mathcal{R}_1(j)|X'$,
   $s\cup u$ and $t\cup v$ may be mixed by $B$ only if $u$ and $v$ are subtrees of the same $X'(l)$ for some $l$.
For each pair of trees  $S,T\in\mathcal{T}(k)$
such that 
$\pi_S(\tilde{T}(k))\in\mathcal{R}_1(i)$ and $\pi_T(\tilde{T}(k))\in\mathcal{R}_1(j)$,
let
\begin{equation}
\mathcal{X}_{S,T}=\{Y\le X': B\mathrm{\ mixes\ } s\cup\pi_{S}(Y(k))\mathrm{\ and\ } t\cup\pi_{T}(Y(k))\}.
\end{equation}
By finitely many applications of the Abstract Ellentuck Theorem,
we may thin to a $Y\le X'$ which is homogeneous for $\mathcal{X}_{S,T}$ for each such  pair $S,T$.

\begin{subclaimn}
There is a $Y'\le Y$ such that for each  pair 
$S,T\in\mathcal{T}(k)$
such that 
$\pi_S(\tilde{T}(k))\in\mathcal{R}_1(i)$ and $\pi_T(\tilde{T}(k))\in\mathcal{R}_1(j)$,
and  each $Z\le Y'$,
if $\vp_{s}(\pi_{S}(Z(k)))\ne\vp_{t}(\pi_{T}(Z(k)))$,
then
$[\emptyset,Z]\cap\mathcal{X}_{S,T}=\emptyset$.
\end{subclaimn}

Suppose not.
Then there is such  a pair $S,T$ such that for each $Y'\le Y$, there is a $Z\le Y'$ such that
$\vp_{s}(\pi_{S}(Z(k)))\ne\vp_{t}(\pi_{T}(Z(k)))$,
but
$[\emptyset,Z]\cap\mathcal{X}_{S,T}\ne\emptyset$.
Recall that 
$\vp_{s}(\pi_{S}(Z(k)))= \pi_{T_s}\circ\pi_{S}(Z(k))$ and 
$\vp_{t}(\pi_{T}(Z(k)))=\pi_{T_t}\circ\pi_{T}(Z(k))$.
We may apply the Abstract Ellentuck Theorem  to thin to some $Y'\le Y$ so that
for each $Z\le Y'$, $\pi_{T_s}\circ\pi_{S}(Z(k))\ne\pi_{T_t}\circ\pi_{T}(Z(k))$, but
 $[\emptyset,Y']\sse\mathcal{X}_{S,T}$.
Suppose  there is some
$q\in
\pi_{T_s}\circ\pi_{S}(\tilde{T}(k))\setminus\pi_{T_t}\circ\pi_{T}(\tilde{T}(k))$.
Take $w,w'\in \mathcal{R}_1(k)|Y(l)$  for some $l$
 such that
$w$ and $w'$ differ exactly on their  elements in the place $q$ and any extensions of $q$.
(That is, for each $q'\in\tilde{T}(k)$,
$\pi_{\{q'\}}(w)\ne\pi_{\{q'\}}(w')$ if and only if $q'\sqsupseteq q$.)
Let $u=\pi_{T_s}\circ\pi_{S}(w)$, $u'=\pi_{T_s}\circ\pi_{S}(w')$,
$v=\pi_{T_t}\circ\pi_{T}(w)$,
and $v'=\pi_{T_t}\circ\pi_{T}(w')$.
Then
$u\not\E_s u'$ but
$v\E_t v'$.
Since $[\emptyset,Y']\sse\mathcal{X}_{S,T}$, $B$ mixes $s\cup u$ and $t\cup v$, and $B$ mixes $s\cup u'$ and $t\cup v'$.
$B$ mixes $t\cup v$ and $t\cup v'$, since $v\E_t v'$.
Hence, by transitivity of mixing, $B$ mixes $s\cup u$ and $s\cup u'$, contradicting that $u\not\E_s  u'$.
Likewise, we obtain a contradiction if
there is some
$q\in
\pi_{T_t}\circ\pi_{T}(\tilde{T}(k))\setminus
\pi_{T_s}\circ\pi_{S}(\tilde{T}(k))$.
Therefore, the Subclaim holds.

Since $S,T$ range over all possible such pairs,
 possibly thinning again, there is a $Z\le Y'/(s,t)$ such that the following holds.
For all $u\in\mathcal{R}_1(i)|Z$ and $v\in\mathcal{R}_1(j)|Z$,
 if $s\cup u$ and $t\cup v$ are mixed by $B$, then $\vp_{s}(u)=\vp_t(v)$.
It follows that $T_s$ and $T_t$ must be isomorphic.

Thus, we have shown that there is a $Z\le X$ such that for all $u\in\mathcal{R}_1(i)|Z$ and $v\in\mathcal{R}_1(j)|Z$,
if $Z$ mixes $s\cup u$ and $t\cup v$, then $\vp_s(u)=\vp_t(v)$.
It remains to show that
there is a $C\le Z$ such that for all $u\in\mathcal{R}_1(i)|Z$ and $v\in\mathcal{R}_1(j)|Z$,
 if $\vp_s(u)=\vp_t(v)$, then $Z$ mixes $s\cup u$ and $t\cup v$.

Suppose $S,T\in\mathcal{T}(k)$
is a pair such that
$\pi_S(\tilde{T}(k))\in\mathcal{R}_1(i)$ and $\pi_T(\tilde{T}(k))\in\mathcal{R}_1(j)$,
and
 for all $w\in\mathcal{R}_1(k)|Z$,
$\vp_s(\pi_{S}(w))=\vp_t(\pi_{T}(w))$.
Assume towards a contradiction that $[\emptyset,Z]\cap\mathcal{X}_{S,T}=\emptyset$.
Then for all $w\in \mathcal{R}_1(k)|Z$,
$Z$ separates $s\cup \pi_{S}(w)$ and $t\cup \pi_{T}(w)$.
Let $S',T'$ be any pair in  $\mathcal{T}(k)$
such that $\pi_{S'}(\tilde{T}(k))\in\mathcal{R}_1(i)$ and $\pi_{T'}(\tilde{T}(k))\in\mathcal{R}_1(j)$, and moreover
 such that $\vp_s(\pi_{S'}(x))=\vp_t(\pi_{T'}(x))$ for any (all) $x\in \mathcal{R}_1(k)|Z$.
Then there are $x,y\in \mathcal{R}_1(k)|Z$ such that $\pi_{S}(x)\E_s \pi_{S'}(y)$ and $\pi_{T}(x)\E_t \pi_{T'}(y)$.
%***[This may warrent some explanation.]***
 $Z$ mixes $s\cup \pi_{S}(x)$ and $s\cup \pi_{S'}(y)$, and $Z$ mixes $t\cup \pi_{T}(x)$ and $t\cup \pi_{T'}(y)$.
Thus, $Z$ must separate $s\cup \pi_{S'}(w)$ and $t\cup \pi_{T'}(w)$ for all $w\in \mathcal{R}_1(k)|Z$.

Given any $S',T'$ for which
 $\vp_s(\pi_{S'}(x))\not=\vp_t(\pi_{T'}(x))$,
$Z$ separates $s\cup \pi_{S'}(x)$ and $t\cup \pi_{T'}(x)$.
Thinning again, we obtain a   $Z'\le Z$ which  separates $s$ and $t$,
 contradiction.
 Therefore, $[\emptyset,Z]\sse\mathcal{X}_{S,T}$,
 and thus $Z$ mixes $s\cup \pi_{S}(W(k))$ and $t\cup \pi_{T}(W(k))$ for all $W\le_2 Z$.

Hence, for all pairs $S,T$, we have that
$\vp_s(\pi_{S}(w))=\vp_t(\pi_{T}(w))$ if and only if $[\emptyset,Z]\sse\mathcal{X}_{S,T}$.
Thus,  for all $u\in\mathcal{R}_1(i)|Z$ and $v\in\mathcal{R}_1(j)|Z$,
$Z$ mixes $s\cup u$ and $t\cup v$ if and only if
$\vp_s(u)=\vp_t(v)$.

Finally, we have shown that for all $s,t\in(\hat{\mathcal{F}}\setminus\mathcal{F})|B$ and each $X\le B$, there is a $Z\le X$ such that for all $u\in\mathcal{R}_1(i)|Z$ and $v\in\mathcal{R}_1(j)|Z$,
$Z$ mixes $s\cup u$ and $t\cup v$ if and only if $\vp_s(u)=\vp_t(v)$.
By Lemma \ref{lem.1},
there is a $C\le B$ for which the Claim holds.
\end{proof}

\begin{claim}\label{claim.phisamefsame}
For all $s,t\in\hat{\mathcal{F}}|C$,
if $\vp(s)=\vp(t)$, then $s$ and $t$ are mixed by $C$.
Hence, for all $s,t\in\mathcal{F}|C$,
if $\vp(s)=\vp(t)$, then $f(s)=f(t)$.
\end{claim}

\begin{proof}
Let $s,t\in\hat{\mathcal{F}}|C$,
and suppose that $\vp(s)=\vp(t)$.
It follows that for each $l$, $\vp(s\cap r_l(C))=\vp(t\cap r_l(C))$.

The proof is by induction on $l\le\max(\depth_C(s),\depth_C(t))$.
For $l=0$,
$s\cap r_0(C)=t\cap r_0(C)=\emptyset$, so $C$ mixes $s\cap r_0(C)$ and $t\cap r_0(C)$.
Suppose that $C$ mixes $s\cap r_l(C)$ and $t\cap r_l(C)$.
If $s\cap C(l)=t\cap C(l)=\emptyset$,
then $s\cap r_{l+1}(C)=s\cap r_l(C)$ and $t\cap r_{l+1}(C)=t\cap r_l(C)$; hence  $s\cap r_{l+1}(C)$ and $t\cap r_{l+1}(C)$ are mixed by $C$.
If $s\cap C(l)\ne\emptyset$ and $t\cap C(l)=\emptyset$
then $\vp(s\cap r_{l+1}(C))=\vp(t\cap r_{l+1}(C))$ implies that $T_{r_i(s)}=T_{\lgl\rgl}$, where $i$ is such that $s(i)\sse C(l)$.
By (A4), $r_i(s)=s\cap r_l(C)$ and $r_{i+1}(s)=s\cap r_{l+1}(C)$ are mixed by $C$.
Thus, $s\cap r_{l+1}(C)$ and $t\cap r_{l+1}(C)=t\cap r_l(C)$ are mixed by $C$.
Similarly, if $s\cap C(l)=\emptyset$ and $t\cap C(l)\ne\emptyset$, mixing of $s\cap r_{l+1}(C)$ and $t\cap r_{l+1}(C)$ again follows from (A4).
If both $s\cap C(l)\ne\emptyset$ and $t\cap C(l)\ne \emptyset$,
then by Claim \ref{claim.Tsametype}, $s\cap r_{l+1}(C)$ and $t\cap r_{l+1}(C)$ are mixed by $C$.

By induction,
$s$ and $t$ are mixed by $C$.
In particular, if $s,t\in\mathcal{F}|C$,
then $f(s)=f(t)$.
\end{proof}

\begin{claim}\label{claim.front}
For all $s,t\in\mathcal{F}|C$, $\vp(s)\not\sqsubset \vp(t)$.
\end{claim}

\begin{proof}
Suppose $\vp(s)\sqsubset \vp(t)$.
Let $j$ be maximal such that $\vp(s)=\vp(r_j(t))$.
Then $T_{r_j(t)}\ne T_{\lgl\rgl}$.
Let $l$ be such that $t(j)\sse C(l)$.
Then $r_j(t)=t\cap r_l(C)$, and
$\vp(s\cap r_l(C))=\vp(s)=\vp(r_j(t))=\vp(t\cap r_l(C))$.
 $C$ mixes $s\cap n_l$ and $t\cap r_l(C)$, by Claim \ref{claim.phisamefsame}.
By (A1), $C$ mixes $s\cap r_l(C)$ and $(t\cap r_l(C))\cup v$ for at most one $\E_{r_j(t)}$ equivalence class of $v$'s in $\mathcal{R}_1(j)|C/r_l(C)$.
So there is an $X\le C$ such that $X$ separates $s\cap r_l(C)$ and $t\cap r_l(C)$, contradicting that $s\cap r_l(C)$ and $t\cap r_l(C)$ are mixed by $C$.
\end{proof}

\begin{claim}\label{claim.L}
For all $s,t\in\mathcal{F}|C$, if $f(s)=f(t)$,
then $\vp(s)=\vp(t)$.
\end{claim}

\begin{proof}
Let $s,t\in\mathcal{F}|C$ with $f(s)=f(t)$,
and let $m=\max(\depth_C(s),\depth_C(t))$.
$f(s)=f(t)$ implies that
for all $l\le m$, $C$ mixes $s\cap r_l(C)$ and $t\cap r_l(C)$.
We shall show by induction that for all $l\le m$, $\vp(s\cap r_l(C))=\vp(t\cap r_l(C))$.
For $l=0$, this is clear, so now suppose
 $l<m$ and  $\vp(s\cap r_l(C))=\vp(t\cap r_l(C))$.
If $s\cap C(l)= t\cap C(l)=\emptyset$,
then  $\vp(s\cap r_{l+1}(C))=\vp(s\cap r_l(C))=\vp(t\cap r_l(C))=\vp(t\cap r_{l+1}(C))$.
If both $s\cap C(l)\ne\emptyset$ and $t\cap C(l)\ne\emptyset$,
then by Claim \ref{claim.Tsametype},
$\vp(s\cap r_{l+1}(C))=\vp(t\cap r_{l+1}(C))$.

Finally, suppose that $s\cap C(l)\ne\emptyset$ and $t\cap C(l)=\emptyset$.
Let $i$ be such that $s(i)\sse C(l)$.
If $T_{r_i(s)}\ne T_{\lgl\rgl}$,
then  $t\cap r_{l+1}(C)$ must be a proper initial segment of $t$;
otherwise, we would have $\vp(t)=\vp(t\cap r_{l+1}(C))=\vp(t\cap r_l(C))=\vp(s\cap r_l(C))\sqsubset \vp(s)$, contradicting Claim \ref{claim.front}.
Let $j$ be such that $r_j(t) = t\cap r_{l+1}(C)$.
Then $j<|t|$.
$C$ mixes $r_{j+1}(s) = (s\cap r_l(C))\cup s(i)$ and $r_{j+1}(t)=(t\cap r_l(C))\cup t(j)$;
so $\vp_{r_i(s)}(s(i))=\vp_{r_j(t)}(t(j))$,
by Claim \ref{claim.Tsametype}.
But this contradicts the facts that $T_{r_i(s)}\ne T_{\lgl\rgl}$, $s(i)\sse C(l)$, and $t(j)\cap C(l)=\emptyset$.
It follows that $T_{r_i(s)}$ must be $T_{\lgl\rgl}$;
hence, $\vp(s\cap r_{l+1}(C))=\vp(t\cap r_{l+1}(C))$.
Likewise, if $s\cap C(l)=\emptyset$ and $t\cap C(l)\ne\emptyset$, we find that $\vp(s\cap r_{l+1}(C))=\vp(t\cap r_{l+1}(C))$.
\end{proof}

It remains to show that $\vp$ witnesses that $\R$ is canonical.
By definition, $\vp$ is inner, and by Claim \ref{claim.front},
$\vp$ is Nash-Williams.
By Claims \ref{claim.phisamefsame} and  \ref{claim.L}, we have that
for each $a,b\in\mathcal{F}|C$, $a \R b$ if and only if $\vp(a)=\vp(b)$.
It then follows from Claim \ref{claim.Tsametype}
that $\vp$ is Sperner.
Thus, it only remains to show that $\vp$ is maximal among all inner Nash-Williams maps $\vp'$ on $\mathcal{F}|C$ which also represent the equivalence relation $\R$.
Toward this end, we prove the following Lemma.

\begin{lem}\label{lem.irredssephi}
Suppose $X\le C$ and $\vp'$ is an inner  function on $\mathcal{F}|X$ which represents $\R$.
Then
there is a $Y\le X$ such that for each $t\in\mathcal{F}|Y$,
for each $i<|t|$, there is a tree $S_{r_i(t)}\sse T_{r_i(t)}$
such that
the following hold.
\begin{enumerate}
\item
For each $s\in\mathcal{F}|Y$ for which $s\sqsupset r_i(t)$,
$\vp'(s)\cap s(i)=\pi_{S_{r_i(t)}}(s(i))$.
\item
$\vp'(t)=\bigcup\{\pi_{S_{r_i(t)}}(t(i)):i<|t|\}\sse \vp(t)$.
\end{enumerate}
Thus, $\vp$
 is $\sse$-maximal among all inner functions $\vp'$ on $\mathcal{F}|C$ which represent $\R$.
\end{lem}

\begin{proof}
Let $X\le C$ and   $\vp'$
satisfy the hypotheses.
Note that $\vp'$ is  inner and also represents the equivalence relation $\R$.
For each $t\in\mathcal{F}$, $i<|t|$, and $X'\le X$, since $\vp'$ is inner, by the Abstract Nash-Williams Theorem there is an $X''\le X'$ such that the following holds:
There is a tree $S_{r_i(t)}\in\mathcal{T}(i)$
such that for each $s\in \mathcal{F}$ extending $r_i(t)$ with $s\setminus r_i(t)\in\Ext(X'')$,
$\vp'(s)\cap s(i)=\pi_{S_{r_i(t)}}(s(i))$.
By Lemma \ref{lem.1},
there is a $Y\le X$ such that for each $t\in\mathcal{F}|Y$ and each $i<|t|$,
there is a tree $S_{r_i(t)}$ satisfying (1).
Thus,  for each $t\in\mathcal{F}|Y$,
\begin{equation}
\vp'(t)=\bigcup\{\pi_{S_{r_i(t)}}(t(i)):i<|t|\}.
\end{equation}

Note that each $S_{r_i(t)}$ must be contained within $T_{r_i(t)}$, the tree from Theorem \ref{thm.PRR(1)} associated with $\E_{r_i(t)}$-mixing of immediate extensions of $r_i(t)$.
Otherwise, there would be $u,v\in\mathcal{R}_1(i)|Y/r_i(t)$ such that $r_i(t)\cup u$ and $r_i(t)\cup v$ are mixed, yet all extensions of them have different $\vp'$ values, which would contradict that $\vp'$ induces the same equivalence relation as $f$.
Thus, for each $t\in\mathcal{F}|Y$,
$\vp'(t)\sse\vp(t)$.
\end{proof}

By
Lemma \ref{lem.irredssephi}, $\R$ is canonical on $\mathcal{F}|C$, which  finishes the proof of the theorem.
\end{proof}

\begin{rem}
The map $\vp$ from Theorem \ref{thm.PRR(1)} has the following property.
One can thin to a $Z$  such that
\begin{enumerate}
\item[$(*)$]
for each
$s\in\mathcal{F}|Z$, there is a $t\in\mathcal{F}$ such that $\vp(s)=\vp(t)=s\cap t$.
\end{enumerate}
This is not the case for any smaller inner map $\vp'$, by Lemma \ref{lem.irredssephi}.
For  suppose $\vp'$ is an inner map representing $\R$,  $\vp'$ satisfies  the conclusions of Lemma \ref{lem.irredssephi} on $\mathcal{F}|Y$, and there is an $s\in\mathcal{F}|Y$ for which $\vp'(s)\subsetneq \vp(s)$.
Then there is some $i<|s|$ for which the tree $S_{r_i(s)}\subsetneq T_{r_i(s)}$.
This implies that $\vp'(t)\subsetneq \vp(t)$ for every $t\in\mathcal{F}|Y$ such that $t\sqsupset r_i(s)$.
Recall that $\vp'(t)=\vp'(s)$ if and only if $\vp(t)=\vp(s)$; and in this case, $\vp(t)\cap\vp(s)\sse t\cap s$.
It follows that for any $t$ for which $\vp'(t)=\vp'(s)$,
$\vp'(t)\cap\vp'(s)$ will always be a proper subset of $t\cap s$.
Thus, $\vp$ is the minimal inner map for which property $(*)$ holds.

It may also be of interest to note that for $\vp'$ inner and  $s\in\mathcal{F}|Z$ from Lemma \ref{lem.irredssephi}, 
if $i<|s|$ is maximal such that 
$T_{r_i(s)}\ne T_{\lgl\rgl}$,
then $i$ is also maximal such that $S_{r_i(s)}\ne T_{\lgl\rgl}$, and moreover, $S_{r_i(s)}=T_{r_i(s)}$.
\end{rem}

\begin{example}
Let $\mathcal{F}$ be the analogue of the Shreier barrier for $\mathcal{R}_1$.
That is, enumerating the elements of $\mathcal{R}_1(0)$ as $\{a_n:n<\om\}$,
 $\mathcal{F}_{a_n}$, the collection of all $t\in\mathcal{F}$ such that $t(0)=a_n$, is isomorphic to $\mathcal{AR}_{n}$.
Let $\R$ be the equivalence relation on $\mathcal{F}$, where  $s\R t$ if and only if $|t|=|s|$ and $t(|t|-1)=s(|s|-1)$.
Then the map $\vp$ from Theorem \ref{thm.PRR(1)} for $\R$ has the property that $\vp(t)\cap t(0)=t(0)$ for all $t\in\mathcal{F}$.

The following map $\vp'$  is  inner Nash-Williams and also represents the  equivalence relation $\R$.
Let $\vp'(t)= t(|t|-1)$, for each $t$ in $\mathcal{F}$.
Then $\vp'(t)\subsetneq \vp(t)$ for all $t\in\mathcal{F}$.
However, $\vp'$  does not satisfy the property $(*)$.
\end{example}
\vskip.1in

We now prove Theorem \ref{thm.original}.
\vskip.1in

\begin{proof}
\it (Theorem \ref{thm.original}). \rm
Let $1\le n<\om$ and $\R$ be an equivalence relation on $\mathcal{AR}_n$.
Let $f:\mathcal{AR}_n\ra\bN$ be any function which induces the equivalence relation $\R$.
Let $C\le A$ be obtained from Theorem \ref{thm.PRR(1)}.
Then for each $s\in\mathcal{AR}_n|C$, there is a sequence $\lgl T_{r_i(s)}:i<n\rgl$ of trees, where each $T_{r_i(s)}\in\mathcal{T}(i)$,
satisfying the following.
For each $s,t\in\mathcal{AR}_n|C$,
$f(s)=f(t)$ if and only if $\bigcup_{i<n}\pi_{T_{r_i(s)}}(s(i))=\bigcup_{i<n}\pi_{T_{r_i(t)}}(t(i))$.
We shall apply the Abstract Ellentuck Theorem to obtain a $D\le C$ such that for all $s,t\in\mathcal{AR}_n|D$ and all $i<n$, $T_{r_i(s)}=T_{r_i(t)}$.
By  Theorem \ref{thm.PRR(1)},
for all $s,t\in\mathcal{AR}_n|C$, $T_{r_0(s)}=T_{r_0(t)}$, so let $X_0=C$ and $T(0)=T_{r_0(s)}$ for any (all) $s\in\mathcal{AR}_n|C$.
Given $i<n-1$, $X_i$, and $T(i)$, then
for each $T\in\mathcal{T}(i+1)$, define
\begin{equation}
\mathcal{X}_T=\{X\le C: T_{r_{i+1}(X)}=T\}.
\end{equation}
The open sets $\mathcal{X}_T$, $T\in\mathcal{T}(i+1)$, cover $[\emptyset,C]$, so there is some $T(i+1)\in\mathcal{T}(i+1)$ and some  $X_{i+1}\le X_i$ such that $[\emptyset,X_{i+1}]\sse\mathcal{X}_{T(i+1)}$.

Let $D=X_{n-1}$.
Then for all $s,t\in\mathcal{AR}_n|D$,
\begin{align}
f(s)=f(t) &\Leftrightarrow \vp(s)=\vp(t)\cr
&\Leftrightarrow
\forall i<n,\
\pi_{T_{r_i(s)}}(s(i))=\pi_{T_{r_i(t)}}(t(i))\cr
&\Leftrightarrow
\forall i<n,\
\pi_{T(i)}(s(i))=\pi_{T(i)}(t(i))\cr
&\Leftrightarrow
\forall i<n,\
s(i)\E_{T(i)} t(i).
\end{align}
Thus, the equivalence relation induced by $f$ is canonical on $\mathcal{AR}_n|D$.
\end{proof}

\begin{cor}\label{cor.canonR(n)}
Let $A\in\mathcal{R}_1$, $1\le n<\om$, and $\E$ be an equivalence relation on $\mathcal{R}_1(n)|A$.
Then there is a $C\le A$ and a tree $T\in\mathcal{T}(n)$ such that for all $a,b\in\mathcal{R}_1(n)|C$,
\begin{equation}
a\E b\Leftrightarrow \pi_T(a)=\pi_T(b).
\end{equation}
\end{cor}

%************************************************************************************************************************************************************************************************************

\section{The Tukey ordering below $\mathcal{U}_1$ in terms of the Rudin-Keisler ordering}\label{sec.R1Tukey}

The canonization theorem from the previous section will now be applied to characterize all ultrafilters which are Tukey reducible to $\mathcal{U}_1$.
Every topological Ramsey space has its own notion of a  Ramsey and selective ultrafilters (see \cite{MR2330595}).
We strengthen the  definition of Ramsey ultrafilter from \cite{MR2330595} to (2) below.

\begin{defn}\label{defn.RamseyufU1}
\begin{enumerate}
\item
We shall say that a subset
$\mathcal{C}\sse\mathcal{R}_1$ {\em  satisfies the Abstract Nash-Williams Theorem} if and only if for each family $\mathcal{G}\sse\mathcal{AR}$ and partition $\mathcal{G}=\mathcal{G}_0\cup\mathcal{G}_1$,
there is a $C\in\mathcal{C}$ and an $i\in 2$ such that  $\mathcal{G}_i|C=\emptyset$.
\item
An ultrafilter $\mathcal{U}_1$ defined on the base set $\mathbb{T}$ is called {\em Ramsey for $\mathcal{R}_1$}
if and only if $\mathcal{U}_1$ is generated by a subset  $\mathcal{C}\sse\mathcal{R}_1$ 
which satisfies the Abstract Nash-Williams Theorem.
\item
An ultrafilter generated by a set $\mathcal{C}\sse\mathcal{R}_1$ is {\em selective for $\mathcal{R}_1$} if and only if
for each decreasing sequence $X_0\ge X_1\ge\dots$ of members of $\mathcal{C}$, there is another $X\in\mathcal{C}$ such that for each $n<\om$, $X\le X_n/r_n(X_n)$.
\end{enumerate}
\end{defn}

Ultrafilters which are Ramsey for $\mathcal{R}_1$ exist, assuming  CH or  MA, or forcing with $(\mathcal{R}_1,\le^*)$.
Since $\mathcal{R}_1$ is isomorphic to a dense subset of Laflamme's forcing $\bP_1$ in \cite{Laflamme89},
any ultrafilter $\mathcal{U}_1$ forced by $(\mathcal{R}_1,\le^*)$ is isomorphic to 
an ultrafilter under the same name forced  by $(\bP_1,\le_{\bP_1}^*)$.

The following facts are straightforward.  (2) is a consequence of Lemma 3.8 in \cite{MR2330595}.
We shall say that   $\mathcal{F}\sse\mathcal{AR}$ is a {\em front on a set $\mathcal{C}\sse\mathcal{R}_1$}
if $\mathcal{F}$ is Nash-Williams, and for each $X\in\mathcal{C}$, there is an $a\in\mathcal{F}$ such that $a\sqsubset X$.

\begin{fact}
\begin{enumerate}
\item
If $\mathcal{U}_1$ is Ramsey for $\mathcal{R}_1$ generated by a set $\mathcal{C}\sse\mathcal{R}_1$, then for each front $\mathcal{F}$ on $\mathcal{C}$ and each 
$\mathcal{G}\sse\mathcal{F}$,
there is a $U\in\mathcal{C}$ such that either $\mathcal{F}|U\sse\mathcal{G}$, or else $\mathcal{F}|U\cap\mathcal{G}=\emptyset$.
\item
Any ultrafilter Ramsey for $\mathcal{R}_1$ is also  selective for $\mathcal{R}_1$.
\end{enumerate}
\end{fact}

We now fix the following notation for the rest of this section.

\begin{notn}\label{notn.ufonafront}
Let  $\mathcal{U}_1$ denote any ultrafilter on base set $\mathbb{T}$ which is Ramsey for $\mathcal{R}_1$ and such that for any front $\mathcal{F}$ on $\mathcal{R}_1$ and any equivalence relation $\R$ on $\mathcal{F}$,
there is a $U\in\mathcal{U}_1\cap\mathcal{R}_1$ such that 
$\R$ is canonical on $\mathcal{F}|U$.

Let $\mathcal{C}$ denote $\mathcal{U}_1\cap\mathcal{R}_1$.
Then $\mathcal{C}$ is
cofinal in $\mathcal{U}_1$.
For any front $\mathcal{F}$ on $\mathcal{C}$ and
  any $X\in\mathcal{C}$,
recall that
 $\mathcal{F}| X$ denotes
$\{a\in\mathcal{F}:a\le_{\mathrm{fin}} X\}$.
Let
\begin{equation}
\mathcal{C}\re\mathcal{F}
=\{\mathcal{F}| X:X\in\mathcal{C}\}.
\end{equation}
\end{notn}

\begin{fact}\label{fact.ufonafront}
Let $\mathcal{B}$ be any cofinal subset of $\mathcal{C}$, and let $\mathcal{F}\sse\mathcal{AR}$ be any front on $\mathcal{B}$.
Then
$\mathcal{B}\re\mathcal{F}$ generates an ultrafilter on $\mathcal{F}$.
\end{fact}

\begin{proof}
For every pair $X,Y\in\mathcal{B}$,
there is a $Z\in\mathcal{B}$ such that $Z\le X,Y$.
Thus, $\mathcal{F}| Z\sse\mathcal{F}| X\cap\mathcal{F}| Y$.
Hence, $\mathcal{B}\re\mathcal{F}$ has the finite intersection property.

Let $\mathcal{G}\sse\mathcal{F}$ 
 and  $X\in\mathcal{B}$.
Since $\mathcal{U}_1$ is Ramsey  for $\mathcal{R}_1$, 
there is a $Y\in\mathcal{C}$ such that $Y\le X$ and  either $\mathcal{F}|Y\sse \mathcal{G}$ or else $\mathcal{F}|Y\cap\mathcal{G}=\emptyset$.
Since $\mathcal{B}$ is cofinal in $\mathcal{C}$, there is a
$Z\in\mathcal{B}$ with $Z\le Y$
such that  
  either $\mathcal{F}|Z\sse \mathcal{G}$ or else $\mathcal{F}|Z\cap\mathcal{G}=\emptyset$.
In the first case, $\mathcal{G}\in\mathcal{B}\re\mathcal{F}$, and in the second case, $\mathcal{F}\setminus\mathcal{G}\in\mathcal{B}\re\mathcal{F}$.
Hence, $\mathcal{B}\re\mathcal{F}$ generates an ultrafilter on $\mathcal{F}$.
\end{proof}

\begin{fact}\label{fact.ufequal}
Suppose $\mathcal{U}$ and $\mathcal{V}$ are proper ultrafilters on the same countable base set, and for each $V\in\mathcal{V}$ there is a $U\in\mathcal{U}$ such that $U\sse V$.
Then $\mathcal{U}=\mathcal{V}$.
\end{fact}

\begin{proof}
Without loss of generality, suppose the base set of $\mathcal{U}$ and $\mathcal{V}$ is $\om$.
Suppose that there is a  $U\in\mathcal{U}\setminus\mathcal{V}$.
Then $\om\setminus U\in\mathcal{V}$.
By hypothesis, there is a $U'\in\mathcal{U}$ such that $U'\sse \om\setminus U$; contradiction to $\mathcal{U}$ being a proper filter.
If there is a $V\in\mathcal{V}\setminus\mathcal{U}$,
 then by hypothesis, there is a $U\in\mathcal{U}$ such that $U\sse V$.
 But $\om\setminus V\in\mathcal{U}$, contradicting that $\mathcal{U}$ is a proper filter.
Thus, the fact holds.
\end{proof}

Recall that by Theorem \ref{thm.5}, every Tukey reduction from a p-point to another ultrafilter is witnessed by a continuous cofinal map.
The proof of  Theorem \ref{thm.5} actually gives more.
The continuous monotone cofinal map $g:\mathcal{P}(\bN)\ra\mathcal{P}(\bN)$ has the additional properties:
There is a function $\hat{g}:2^{<\om}\ra \mathcal{P}(\om)$ such
that, for any $X\sse\bN$,
identifying $X\cap k$ with its characteristic function with domain $k$, we have
\begin{enumerate}
\item
For each $k\in\bN$ and each $s\in 2^k$,
$\hat{g}(s)\sse k$;
\item
$s\sqsubseteq t\in 2^{<\om}$ implies
$\hat{g}(s)\sqsubseteq \hat{g}(t)$;
\item
For each $X\sse\bN$,
$g(X)=\bigcup_{k<\om}\hat{g}(X\cap k)$;
and
\item
For each $X\sse\bN$ and $k\in\bN$,
$g(X)\cap k=\hat{g}(X\cap k)$;
\item
$\hat{g}$ is monotonic;
that is,
if $k\le m\in\bN$, $s\in 2^k$, and $t\in 2^m$ are such that $s$ and $t$ are characteristic functions for  sets $x,y\sse\bN$, respectively, with $x\sse y$,
then,
 $\hat{g}(s)\sse\hat{g}(t)$.
\end{enumerate}

\begin{prop}\label{prop.W=frontuf}
Suppose $\mathcal{V}$ is a nonprincipal ultrafilter (without loss of generality on $\bN$) such that  $\mathcal{U}_1\ge_T\mathcal{V}$.
Then there is a front $\mathcal{F}$ on $\mathcal{C}$ and a function $f:\mathcal{F}\ra\bN$ such that
$\mathcal{V}=f(\lgl\mathcal{C}\re\mathcal{F}\rgl)$.
\end{prop}

\begin{proof}
By Theorem \ref{thm.5},
there is a continuous monotone cofinal map $g:\mathcal{U}_1\ra\mathcal{V}$ which is given by a monotone function $\hat{g}:2^{<\om}\ra \mathcal{P}(\om)$.
Define $\mathcal{F}\sse\mathcal{AR}$ to consist of all $r_n(X)$ such that $X\in\mathcal{C}$ and $n$ is minimal such that $\hat{g}(r_n(X))\ne\emptyset$.
Then $\mathcal{F}$ forms a front on  $\mathcal{C}$.
By Fact \ref{fact.ufonafront}, $\mathcal{C}\re\mathcal{F}$ generates an ultrafilter on the front $\mathcal{F}$ as a base set.
Define $f:\mathcal{F}\ra\bN$ by $f(a)=\min(\hat{g}(a))$, for $a\in\mathcal{F}$.
Note that $f(a)=\min(g(X))$ for any $X\in\mathcal{C}$ for which $a\sqsubset X$.
For each $X\in\mathcal{C}$,
$f(\mathcal{F}| X)=\{f(a):a\in\mathcal{F}| X\}$.
Since $\mathcal{C}\re\mathcal{F}$ generates an ultrafilter,
  its Rudin-Keisler image under $f$,  $f(\lgl\mathcal{C}\re\mathcal{F}\rgl)$, is an ultrafilter on $\bN$.

\begin{claim}\label{claimn.5}
If $\mathcal{V}$ is nonprincipal, then $f(\mathcal{F}|X)$ is infinite, for each $X\in\mathcal{C}$.
Hence, $f(\lgl\mathcal{C}\re\mathcal{F}\rgl)$ is a nonprincipal ultrafilter.
\end{claim}

\begin{proof}
Suppose $\mathcal{V}$ is non-principle.
Since $\mathcal{C}$ is a cofinal subset of $\mathcal{U}_1$ and the $g$-image of $\mathcal{C}$ is cofinal in $\mathcal{V}$,
we have that $\mathcal{V}$ equals the filter generated by the $g$-image of $\mathcal{C}$.
It follows that for all $X\in\mathcal{C}$ and $k$,
$g(X)\setminus k$ is also in $\mathcal{V}$.
Therefore, there is a $Y\in\mathcal{C}$ such that $g(Y)\sse g(X)\setminus k$.
Hence, for $n$ such that $r_n(Y)\in\mathcal{F}$, we have that $f(r_n(Y))=\min(g(Y))\ge k$.
Since $\mathcal{F}| X$ contains $r_n(Y)$ for each $Y\in\mathcal{C}$ such that $Y\le X$,
it follows that
$f$ takes on infinitely many values on $\mathcal{F}| X$, so $f(\mathcal{F}| X)$ must be infinite.
Moreover, for each $k$, there is an $X\in\mathcal{C}$ such that $k\le\min(g(X))$; so $k\cap f(\mathcal{F}| X)=\emptyset$.
Therefore, the ultrafilter generated by $f(\lgl\mathcal{C}\re\mathcal{F}\rgl)$ contains the Fr\'{e}chet filter.
Thus, $f(\lgl\mathcal{C}\re\mathcal{F}\rgl)$ is a nonprincipal ultrafilter.
\end{proof}

If $\mathcal{V}$ is nonprincipal,  then by Claim \ref{claimn.5},
$f(\lgl\mathcal{C}\re\mathcal{F}\rgl)$ is a nonprincipal  ultrafilter.
Note that for each $X\in\mathcal{C}$,
$f(\mathcal{F}| X)\sse g(X)$.
Since both $\mathcal{V}$ and $f(\lgl\mathcal{C}\re\mathcal{F}\rgl)$ are nonprincipal ultrafilters, they must be equal, by Fact \ref{fact.ufequal}.
In fact, the upwards closure of $\{f(\mathcal{F}|X):X\in\mathcal{C}\}$ is exactly $\mathcal{V}$.
\end{proof}

There is a Rudin-Keisler increasing chain of ultrafilters associated with the space $\mathcal{R}_1$, for which we now fix some notation.

\begin{notn}\label{notn.Phi_n}
Recall that $\mathcal{R}_1(n)|X$ denotes the collection
$\{Y(n):Y\le X\}$.
\begin{enumerate}
\item
For each $n<\om$,
define $\mathcal{U}_1|\mathcal{R}_1(n)$ to be the filter on the base $\mathcal{R}_1(n)$ generated by the sets $\mathcal{R}_1(n)|X$, $X\in\mathcal{C}$.
To make notation more concise, let $\mathcal{Y}_{n+1}$ denote $\mathcal{U}_1|\mathcal{R}_1(n)$.
\item
Define $\mathcal{U}_0=\pi_{T_{\lgl 0\rgl}}(\mathcal{U}_1)$, and
let $\mathcal{Y}_0=\pi_{T_{\lgl 0\rgl}}(\mathcal{Y}_1)$.
\end{enumerate}
\end{notn}

The subtle difference between $\mathcal{U}_0$ and $\mathcal{Y}_0$ is that 
$\mathcal{U}_0$ has as its base the set
 $\{\lgl\rgl\}\cup\{\lgl n\rgl:n<\om\}$, 
  whereas 
the base for $\mathcal{Y}_0$ is $\{\{\lgl\rgl,\lgl n\rgl\}:n<\om\}$.
Likewise, the base for $\mathcal{U}_1$ is $\bT$, whereas the base for $\mathcal{Y}_1$ is $\mathcal{R}_1(0)$.
We point out the following fact, as it clarifies the relationships between the ultrafilters $\mathcal{U}_0$, $\mathcal{U}_1$,
and the $\mathcal{Y}_n$, $n<\om$.

\begin{fact}\label{fact.iso}
\begin{enumerate}
\item
$\mathcal{U}_0$ is the ultrafilter  generated by the sets $\{\lgl\rgl\}\cup\{\lgl j\rgl : \lgl j\rgl\in X\}$, $X\in\mathcal{C}$.
\item
$\mathcal{U}_0\cong\mathcal{Y}_0$.
Moreover, $\pi_{T_{\lgl 0\rgl}}(\mathcal{Y}_n)=\mathcal{Y}_0$, for any $n<\om$.
\item
 $\mathcal{U}_1\cong\mathcal{Y}_1$.
\item
For any $m<n$ and $T\in\mathcal{T}(n)$ such that $T\cong \tilde{T}(m)$,
$\pi_T(\mathcal{Y}_n)=\mathcal{Y}_m$.
\end{enumerate}
\end{fact}

\begin{prop}\label{prop.structureU_R(n)}
\begin{enumerate}
\item
$\mathcal{U}_0$ is a Ramsey ultrafilter.
\item
$\mathcal{U}_1$ is a weakly Ramsey ultrafilter which is not Ramsey, and which satisfies the $(1,k)$ Ramsey partition property for each $k\ge 1$.
\item
For each $n\ge 2$,
 $\mathcal{Y}_n$ is an ultrafilter, and moreover is a  rapid p-point.
\item
$\mathcal{U}_0<_{RK}\mathcal{U}_1<_{RK}\mathcal{Y}_2<_{RK}\mathcal{Y}_3<_{RK}\dots$.
\item
For each $n\ge 1$, $\mathcal{Y}_n\equiv_T\mathcal{U}_1$.
\end{enumerate}
\end{prop}

\begin{proof}
Since $\mathcal{R}_1$ is dense in Laflamme's forcing $\bP_1$, (1) and (2) follow  from  Theorem
\ref{thm.Laflammethms}.

(3) Let $n\ge 2$.
It is clear that $\mathcal{Y}_n$ is a filter.
Let $V$ be any subset of  $\mathcal{R}_1(n-1)$, and
let $\mathcal{H}=\{a\in\mathcal{AR}_{n}: a(n-1)\in V\}$.
Since  $\mathcal{U}_1$ is Ramsey for $\mathcal{R}_1$, there is an $X\in\mathcal{C}$ such that either
$\mathcal{AR}_{n}|X\sse\mathcal{H}$ or else
$\mathcal{AR}_{n}|X\cap\mathcal{H}=\emptyset$.
In the first case, $V\in\mathcal{Y}_n$ and in the second case, $\mathcal{R}_1(n-1)\setminus V\in\mathcal{Y}_n$.
Thus, $\mathcal{Y}_n$ is an ultrafilter.

Suppose $U_0\contains U_1\contains\dots$ is a decreasing sequence of elements of $\mathcal{Y}_n$.
For each $k<\om$, there is some $X_k\in\mathcal{R}_1$ for which $\mathcal{R}_1(n-1)|X_k\sse U_k$.
We may take $(X_k)_{k<\om}$ to be a $\le$-decreasing sequence.
Since $\mathcal{U}_1$ is selective for $\mathcal{R}_1$, there is an $X\in\mathcal{C}$ such that $X/r_k(X)\le X_k$, for each $k<\om$.
Then $\mathcal{R}_1(n-1)|X\sse^*\mathcal{R}_1(n-1)|X_k$, for each $k<\om$.
Thus, $\mathcal{Y}_n$ is a p-point.

To show that $\mathcal{Y}_n$ is rapid, let $h:\om\ra\om$ be a strictly increasing function.
Linearly order $\mathcal{R}_1(n-1)$ so that all members of $\mathcal{R}_1(n-1)|\mathbb{T}(k)$ appear before all members of $\mathcal{R}_1(n-1)|\mathbb{T}(k+1)$ for all $k\ge n-1$.
For any tree $u$, let $\min(\pi_{T_{\lgl 0\rgl}}(u))$ denote the smallest $l$ such that $\lgl l\rgl\in \pi_{T_{\lgl 0\rgl}}(u)$.
For each $X\in\mathcal{R}_1$, there is a $Y\le X$ such that
$\min(\pi_{T_{\lgl 0\rgl}}(Y(n-1)))>h(1)$, $\min(\pi_{T_{\lgl 0\rgl}}(Y(n)))>h(1+|\mathcal{R}_1(n-1)|\mathbb{T}(n)|)$,
 and in general, for $k>n$,
\begin{equation}
\min(\pi_{T_{\lgl 0\rgl}}(Y(k)))
>h(\Sigma_{n\le i\le k} |\mathcal{R}_1(n-1)|\mathbb{T}(i)|).
\end{equation}
Since $\mathcal{U}_1$  is selective for $\mathcal{R}_1$, there is a $Y\in\mathcal{C}$ with this property, which yields that
$\mathcal{Y}_n$ is rapid.

(4)
First,
$\mathcal{Y}_0\cong\mathcal{U}_0\le_{RK}\mathcal{U}_1\cong\mathcal{Y}_1$.
Now suppose  $1\le n<\om$.
$\mathcal{Y}_n\le_{RK}\mathcal{Y}_{n+1}$ is witnessed by the map $\pi_{\tilde{T}(n-1)}:\mathcal{R}_1(n)\ra\mathcal{R}_1(n-1)$,
since $\pi_{\tilde{T}(n-1)}(\mathcal{Y}_{n+1})=\mathcal{Y}_n$.
Next we show that the only Rudin-Keisler predecessors of $\mathcal{Y}_n$ are isomorphic to $\mathcal{Y}_k$ for some $k\le n$, and that $\mathcal{Y}_{n+1}\not\le_{RK}\mathcal{Y}_n$.
Let $\theta:\mathcal{R}_1(n-1)\ra\bN$ be any function.
By the Corollary \ref{cor.canonR(n)} to the Canonization Theorem
and  $\mathcal{U}_1$ being Ramsey for $\mathcal{R}_1$,
there is an $X\in\mathcal{C}$
and a subtree $T\sse\tilde{T}(n-1)$
such that for all $Y,Z\in\mathcal{C}|X$,
$\theta(Y(n-1))=\theta(Z(n-1))$ iff $Y(n-1) \E_T Z(n-1)$.
It follows that $\theta(\mathcal{Y}_n)$ is isomorphic to $\mathcal{Y}_k$ for some $k\le n$.

Similarly,
if we let $\theta:\mathcal{R}_1(n-1)\ra\mathcal{R}_1(n)$ be any function,
by the Canonization Theorem and $\mathcal{U}_1$ being Ramsey for $\mathcal{R}_1$,
there is an $X\in\mathcal{C}$
and a subtree $T\sse\tilde{T}(n-1)$
such that for all $Y,Z\in\mathcal{C}|X$,
$\theta(Y(n-1))=\theta(Z(n-1))$ iff $Y(n-1) \E_T Z(n-1)$.
It follows that $\theta(\mathcal{Y}_n)\ne \mathcal{Y}_{n+1}$.

(5)
Let $n>1$.
Define a  map $g:\mathcal{Y}_n|\mathcal{C}\ra\mathcal{C}$ by
$g(\mathcal{R}_1(n-1)|X)=X$, for each $X\in\mathcal{C}$.
$g$ is well-defined, since from the set $\mathcal{R}_1(n-1)|X$ one can unambiguously reconstruct $X$.
Thus, $g$ is a monotone cofinal map from a cofinal subset of $\mathcal{Y}_n$ into a cofinal subset of $\mathcal{U}_1$, so $g$ witnesses that $\mathcal{U}_1\le_T\mathcal{Y}_n$.
On the other hand, $\mathcal{Y}_n$ is generated by the image of the monotone cofinal map $g:\mathcal{C}\ra\mathcal{R}_1(n-1)|\mathcal{C}$ defined by
$g(X)=\mathcal{R}_1(n-1)|X$.
Thus, $\mathcal{Y}_n\le_T\mathcal{U}_1$.
Therefore, $\mathcal{Y}_n\equiv_T\mathcal{U}_1$.
\end{proof}

\begin{rem}\label{rem.Un+1Un}
In fact, (4) in the above theorem will be strengthened: It will follow from Theorem \ref{thm.TukeyU_1} that, for each $n<\om$, the only nonprincipal ultrafilters Rudin-Keisler reducible to $\mathcal{Y}_n$ are those which are isomorphic to $\mathcal{Y}_k$ for some $k\le n$.
Thus, the ultrafilters $\mathcal{U}_0<_{RK}\mathcal{U}_1<_{RK}\mathcal{Y}_2<_{RK}\dots$ form a maximal chain of isomorphism types 
among all nonprincipal ultrafilters with Tukey type less than or equal to the Tukey type of $\mathcal{U}_1$.
\end{rem}

\begin{thm}\label{thm.TukeyU_1}
Suppose $\mathcal{U}_1$ is Ramsey for $\mathcal{R}_1$ and
$\mathcal{V}$ is a nonprincipal ultrafilter and  $\mathcal{U}_1\ge_T\mathcal{V}$.
Then $\mathcal{V}$ is isomorphic to an ultrafilter of  $\vec{\mathcal{W}}$-trees, where
$\hat{\mathcal{S}}\setminus\mathcal{S}$ is a well-founded tree, $\vec{\mathcal{W}}=(\mathcal{W}_s:s\in\hat{\mathcal{S}}\setminus\mathcal{S})$, 
  and each $\mathcal{W}_s$ is exactly one of the   $\mathcal{Y}_n$, $n<\om$.
\end{thm}

\begin{proof}
The proof is structured as follows.
We will show there is a front $\mathcal{F}$ on $\mathcal{C}$,  a function $f:\mathcal{F}\ra\bN$, and a $C\in\mathcal{C}$ such that, letting  $\mathcal{S}=\{\vp(t):t\in\mathcal{F}|C\}$, the following hold.
\begin{enumerate}
\item
The equivalence relation induced by $f$ on $\mathcal{F}|C$ is canonical.
\item
$\mathcal{V}=f(\lgl\mathcal{C}\re\mathcal{F}\rgl)$.
\item
$\mathcal{W}$, the filter on base set $\mathcal{S}$ generated by $\vp(\mathcal{C}\re\mathcal{F})$, is an ultrafilter, and $\mathcal{W}\cong\mathcal{V}$.
\item
$\hat{\mathcal{S}}$, the set of all initial segments of elements of $\mathcal{S}$, forms a tree with no infinite branches.
\item
  $\mathcal{W}$ is  the
ultrafilter on $\mathcal{S}$ generated by the  
   $\vec{\mathcal{W}}$-trees,
where  $\vec{\mathcal{W}}=(\mathcal{W}_s:s\in\hat{\mathcal{S}}\setminus\mathcal{S})$, and
for each $s\in\hat{\mathcal{S}}\setminus\mathcal{S}$, the ultrafilter $\mathcal{W}_{s}$ equals  $\mathcal{Y}_n$ for some $n<\om$.
\end{enumerate}

Since $\mathcal{U}_1$ is a p-point, by Theorem \ref{thm.5}
there is a continuous monotone cofinal map $g:\mathcal{P}(\bT)\ra\mathcal{P}(\bN)$ such that $g:\mathcal{U}_1\ra\mathcal{V}$ is a cofinal map.
Moreover, $g\re\mathcal{R}_1$ is produced by a map $\hat{g}:\mathcal{AR}\ra \mathcal{P}(\om)$ of the sort discussed just below Theorem \ref{thm.5}.
Let $\mathcal{F}$ consist of all $r_n(Y)$ such that $Y\in\mathcal{R}_1$ and $n$ is minimal such that $\hat{g}(r_n(Y))\ne\emptyset$.
By the properties of $\hat{g}$,
$\min(\hat{g}(r_n(Y)))=\min(g(Y))$.
By its definition,
$\mathcal{F}$ is a front on $\mathcal{R}_1$, hence is a front on $\mathcal{C}$.
Define a new function $f:\mathcal{F}\ra\bN$ by
$f(b)=\min(\hat{g}(b))$, for each $b\in\mathcal{F}$.
By Theorem \ref{thm.PRR(1)},
for each $X\in\mathcal{R}_1$, there is a $Y\le X$ such that the map $f\re(\mathcal{F}|Y)$ is canonical.

There is a $C\in\mathcal{C}$ such that the equivalence relation induced by $f\re(\mathcal{F}|C)$ is canonical.
For by the construction of $\mathcal{U}_1$, 
given any front $\mathcal{F}'$ and any equivalence relation $\R'$ on $\mathcal{F}'$, 
there is a $Z\in\mathcal{C}$ such that 
 $\R'$ is canonical on $\mathcal{F}'|Z$.
By Proposition \ref{prop.W=frontuf},
$\mathcal{V}=f(\lgl \mathcal{C}\re\mathcal{F}\rgl)$.
If $\mathcal{F}=\{\emptyset\}$, then $\mathcal{V}$ is a principal ultrafilter, so we may assume that $\mathcal{F}\ne\{\emptyset\}$.

From now on we abuse notation and let $\mathcal{F}$ denote $\mathcal{F}|C$ and $\mathcal{C}$ denote $\mathcal{C}|C$.
Let  $\mathcal{S}=\{\vp(t):t\in\mathcal{F}\}$.
%For all $t,t'\in\hat{\mathcal{F}}$,
%$t\sqsubseteq t'$ implies $\vp(t)\sqsubseteq\vp(t')$.
%Moreover, for all $t,t'\in\mathcal{F}$,
%$\vp(t)$ is never a proper initial segment of $\vp(t')$.
Define $\mathcal{W}$ to be the filter on base set $\mathcal{S}$ generated by the sets
$\{\vp(t):t\in \mathcal{F}|X\}$,  $X\in\mathcal{C}$.
For $X\in\mathcal{C}$, let $\mathcal{S}|X$ denote
$\{\vp(t):t\in\mathcal{F}|X\}$.

\begin{claim}\label{claim.thmF,T1}
$\mathcal{W}$ is an ultrafilter.
\end{claim}

\begin{proof}
Given $X,Y\in\mathcal{C}$, there is a $Z\in\mathcal{C}$ such that $Z\le X,Y$; so $\{\vp(t):t\in\mathcal{F}|Z\}\sse\{\vp(t):t\in\mathcal{F}|X\}\cap\{\vp(t):t\in\mathcal{F}|Y\}$.
Thus, $\mathcal{W}$ is a filter.

Let $S\sse\mathcal{S}$ and $X\in\mathcal{C}$ be given.
Let $\mathcal{H}=\{t\in\mathcal{F}:\vp(t)\in S\}$.
Since $\mathcal{U}_1$ is Ramsey for $\mathcal{R}_1$, $\mathcal{C}$ contains  a $Y$ such
that either $\mathcal{F}|Y\sse\mathcal{H}$ or else $\mathcal{F}|Y\cap\mathcal{H}=\emptyset$.
In the first case,
 $\mathcal{S}|Y:=\{\vp(t):t\in\mathcal{F}|Y\}\sse S$;
 so $S\in\mathcal{W}$.
In the second case,  $\mathcal{S}|Y\cap S=\emptyset$; hence
 $\mathcal{S}\setminus S$ is in $\mathcal{W}$.
Therefore, $\mathcal{W}$ is an ultrafilter.
\end{proof}

\begin{claim}\label{claim.W=V}
$\mathcal{W}$ is isomorphic to $\mathcal{V}$.
\end{claim}

\begin{proof}
Define $\theta:\mathcal{S}\ra \om$ by $\theta(\vp(t))=f(t)$, for each $t\in \mathcal{F}$.
Since $f$ is canonical on $\mathcal{F}$,
for all $t,t'\in\mathcal{F}$, $\vp(t)=\vp(t')$ if and only if $f(t)=f(t')$.
Thus, $\theta$ is well-defined.
Moreover, whenever $\theta(\vp(t))=\theta(\vp(t'))$, then $f(t)=f(t')$, which implies $\vp(t)=\vp(t')$;
so $\theta$ is 1-1.

For each $W\in\mathcal{W}$,
there is an $X\in\mathcal{C}$ such that $\mathcal{S}|X\sse W$.
Then $\theta(W)\contains \theta(\mathcal{S}|X) = f(\mathcal{F}|X)\in\mathcal{V}$.
So the image of $\mathcal{W}$ under $\theta$ is contained in $\mathcal{V}$.
Further, the image of $\mathcal{W}$ under $\theta$ is cofinal in $\mathcal{V}$.
For letting $V\in\mathcal{V}$, there is an $X\in\mathcal{C}$ such that $f(\mathcal{F}|X)\sse V$.
Then $\mathcal{S}|X=\{\vp(t):t\in\mathcal{F}|X\}\sse V\in\mathcal{V}$, and moreover,  $\mathcal{S}|X\sse V$.
Thus, $\theta(\mathcal{W})=\mathcal{V}$.
\end{proof}

Let $\hat{\mathcal{S}}$ denote the collection of all initial segments of elements of $\mathcal{S}$.
Precisely,
let $\hat{\mathcal{S}}$ be the collection of all $\vp(t)\cap r_i(t)$ such that $t\in\mathcal{F}$, $i\le |t|$, and if $i<|t|$ then $T_{r_i(t)}\ne T_{\lgl\rgl}$.
 $\hat{\mathcal{S}}$ forms  a tree under the end-extension ordering.

Recall that for $s\in\hat{\mathcal{S}}\setminus\mathcal{S}$,
for all $t,t'\in\mathcal{F}$,
if $j<|t|$ is maximal such that $\vp(r_j(t))=s$ and $j'$ is maximal such that $\vp(r_{j'}(t'))=s$,
then $T_{r_j(t)}$ is isomorphic to $T_{r_{j'}(t')}$, and these are both not $T_{\lgl\rgl}$.
Define $\mathcal{W}_s$ to be the filter
generated by the sets
$\{\vp_{r_j(t)}(u): u\in\mathcal{R}_1(j)|X/t\}$,
for all $t\in\mathcal{F}$ such that $s\sqsubseteq \vp(t)$ and
$j<|t|$ maximal such that $\vp(r_j(t))=s$, and all $X\in\mathcal{C}$.
Note that if $T_{r_j(t)}=T_{\lgl 0\rgl}$, then
the base set for $\mathcal{W}_s$ is  $\{\{\lgl\rgl,\lgl k\rgl\} :k<\om\}$;
and if $T_{r_j(t)}=T_I$, where $0<|I|=n$,
then the base set for $\mathcal{W}_s$ is
 $\mathcal{R}_1(n-1)$.

\begin{claim}\label{claim.thmF,T2}
For each $s\in\hat{\mathcal{S}}\setminus\mathcal{S}$,
$\mathcal{W}_{s}$ is an ultrafilter which is generated by the collection of  $\{\vp_{r_j(t)}(u):u\in\mathcal{R}_1(j)|X\}, X\in\mathcal{C}$, for any (all) $t\in\mathcal{F}$ and $j<|t|$ maximal such that $\vp(r_j(t))=s$.
\end{claim}

\begin{proof}
Let
 $s\in\hat{\mathcal{S}}\setminus\mathcal{S}$.
First we check that $\mathcal{W}_s$ is a nonprincipal filter.
Suppose $t,t'\in\mathcal{F}$ and $j,j'$ are maximal such that $\vp(r_j(t))=\vp(r_{j'}(t'))=s$.
Let $X\in\mathcal{C}$ and let
$S=\{\vp_{r_j(t)}(u):u\in\mathcal{R}_1(j)|X/t\}$ and
$S'=\{\vp_{r_{j'}(t')}(u):u\in\mathcal{R}_1(j')|X/t'\}$.
We claim that $S\cap S'\ne\emptyset$.
Let
\begin{equation}
\mathcal{H}=\{a\in\mathcal{AR}_{j+1}:\exists v\in\mathcal{R}_1(j')|X/(t,t')\, (\vp_{r_j(t)}(a(j))=\vp_{r_{j'}(t')}(v))\}.
\end{equation}
Since  $\mathcal{U}_1$ is Ramsey for $\mathcal{R}_1$,
there is a $Y\le X$  in $\mathcal{C}$ for which either $\mathcal{AR}_{j+1}|Y\sse\mathcal{H}$ or else $\mathcal{AR}_{j+1}|Y \cap \mathcal{H}=\emptyset$.
The second case cannot happen, since for any $Y\le X$, there are $u\in\mathcal{R}_1(j)|Y$ and $v\in\mathcal{R}_1(j')|Y$ for which $\vp_{r_j(t)}(u)=\vp_{r_{j'}(t')}(v)$.
Thus,
$\{\vp_{r_j(t)}(u):u\in\mathcal{R}_1(j)|Y/t\}\sse S'$.
Therefore, $\mathcal{W}_s$ is a nonprincipal filter.
Moreover,
for any $t\in\mathcal{F}$ and $j<|t|$ with $j$ maximal such that $\vp(r_j(t))=s$,
the collection  of sets
$\{\vp_{r_j(t)}(u):u\in\mathcal{R}_1(j)|X/t\}$, $X\in\mathcal{C}$ generates
$\mathcal{W}_s$.
Fix one such $r_j(t)$.

Toward showing that $\mathcal{W}_s$ is an ultrafilter,
let $W\sse\mathcal{S}$.
Let
$\mathcal{H}=\{a\in\mathcal{AR}_{j+1}:\vp_{r_j(t)}(a(j))\in W\}$.
Since $\mathcal{U}_1$ is Ramsey for $\mathcal{R}_1$,
there is a $Y\in\mathcal{C}$ such that either
$\mathcal{AR}_{j+1}|Y\sse\mathcal{H}$ or else $\mathcal{AR}_{j+1}|Y\cap\mathcal{H}=\emptyset$.
In the first case, $\{\vp_{t\re j}(Z(j)):Z\le Y\}\sse W$.
In the second case,
$\{\vp_{r_j(t)}(Z(j)):Z\le Y\}\cap W=\emptyset$.
Since
$\{\vp_{r_j(t)}(u):u\in\mathcal{R}_1(j)|Y\}=\{\vp_{r_j(t)}(Z(j)):Z\le Y\}$,
$\mathcal{W}_s$ is an ultrafilter.
\end{proof}

\begin{claim}\label{claim.thmF,T3}
Let $s\in\hat{\mathcal{S}}\setminus\mathcal{S}$.
Then $\mathcal{W}_{s}$ is isomorphic to $\mathcal{Y}_n$ for some $n<\om$.
\end{claim}

\begin{proof}
Fix $t\in\mathcal{F}$ and $j<|t|$  with $j$ maximal such that $\vp(r_j(t))=s$.
Suppose $T_{r_j(t)}=T_{\lgl 0\rgl}$.
Then for each $X\in\mathcal{C}$,
$\{\vp_{r_j(t)}(u):u\in\mathcal{R}_1(j)|X\}= \pi_{T_{\lgl 0\rgl}}(\mathcal{R}_1(j)|X)\in\mathcal{Y}_0$.
Since $\mathcal{W}_s$ is a nonprincipal ultrafilter,  $\mathcal{W}_s$ must equal  $\mathcal{Y}_0$, by Fact \ref{fact.ufequal}.
If $T_{r_j(t)}=T_I$ and $n=|I|\ge 1$,
then for each $X\in\mathcal{C}$,
$\{\vp_{r_j(t)}(u):u\in\mathcal{R}_1(j)|X\}\sse \mathcal{R}_1(n)|X\in\mathcal{Y}_n$.
Thus, by Fact \ref{fact.ufequal}, $\mathcal{W}_s$ must equal $\mathcal{Y}_n$.
\end{proof}

\begin{claim}\label{claim.W=Wtrees}
$\mathcal{W}$ is  the  ultrafilter of  $\vec{\mathcal{W}}$-trees, where  $\vec{\mathcal{W}}=(\mathcal{W}_s:s\in\hat{\mathcal{S}}\setminus\mathcal{S})$.
\end{claim}

\begin{proof}
Given a tree $\hat{S}\sse\hat{\mathcal{S}}$, let $[\hat{S}]$ denote the set of cofinal branches through $\hat{S}$.
Let 
\begin{equation}
[{\vec{\mathcal{W}}}]
=\{[\hat{S}]:\hat{S}\sse\hat{\mathcal{S}}\mathrm{\  is\ a\ }\vec{\mathcal{W}}\mathrm{-tree}\}.
\end{equation}
We shall show that $\mathcal{W}=[{\vec{\mathcal{W}}}]$.

Let  $X\in\mathcal{C}$, 
$S=\{\vp(t):t\in\mathcal{F}|X\}$, and
 $\hat{S}$ denote the collection of all initial segments of elements of $S$.
Then $S=[\hat{S}]$.
$\hat{S}$ is a $\vec{\mathcal{W}}$-tree:
For each $s\in \hat{S}\setminus\mathcal{S}$,
the set of immediate extensions of $s$ in $\hat{S}$ is
the set of all $\vp_{r_j(t)}(t(j))$ such that $t\in\mathcal{F}|X$, $s\sqsubset \vp(t)$, and $j<|t|$ is maximal such that $\vp(r_j(t))=s$.
This set is an element of $\mathcal{W}_s$.
Further, the set of $\vec{\mathcal{W}}$-trees forms a filter on $\hat{\mathcal{S}}$.
Hence, $[\vec{\mathcal{W}}]$ 
is a nonprincipal filter which contains a cofinal subset of $\mathcal{W}$; thus they are equal.
\end{proof}

Thus, by Claims \ref{claim.W=V} - \ref{claim.W=Wtrees},
$\mathcal{V}$ is isomorphic to the ultrafilter $\mathcal{W}$ on base set $\mathcal{S}$ generated by the $\vec{\mathcal{W}}$-trees, where for each $s\in\hat{\mathcal{S}}\setminus\mathcal{S}$,
$\mathcal{W}_s$ is exactly $\mathcal{Y}_n$ for some $n<\om$.
\end{proof}

\begin{rem}\label{rem.uniform}
Like every topological Ramsey space, there is the usual notion of a uniform front on $\mathcal{R}_1$.
It is routine to show, by induction on rank, that for each $X\in\mathcal{R}_1$ and each front $\mathcal{F}$ on $[\emptyset,X]$,
there is a $Y\le X$ such that $\mathcal{F}|Y$ is uniform.
Thus, Theorem \ref{thm.TukeyU_1} in fact yields that every ultrafilter $\mathcal{V}\le_T\mathcal{U}_1$ is isomorphic to some countable iteration of Fubini products of ultrafilters from among $\mathcal{Y}_n$, $n<\om$.
\end{rem}

\begin{example}[Rudin-Keisler Structure within the Tukey Type of $\mathcal{U}_1$]\label{ex.structures}
The Tukey type of $\mathcal{U}_1$ contains all isomorphism types of countable iterations of Fubini products of $\mathcal{U}_1$.
It follows that the Tukey type of $\mathcal{U}_1$ contains a Rudin-Keisler strictly increasing chain of order type $\om_1$.
It also contains the following Rudin-Keisler strictly increasing chain of rapid p-points of order type $\om$:
$\mathcal{U}_1<_{RK}\mathcal{Y}_2<_{RK}<\mathcal{Y}_3<_{RK}\dots$.
Since each of $\mathcal{U}_1$ and the $\mathcal{Y}_n$, $n\ge 2$, is a p-point, hence none of the ultrafilters in this chain is a Fubini product of any other ultrafilters.
Moreover, it follows from Theorem \ref{thm.TukeyU_1} that this chain is maximal within the Tukey type of $\mathcal{U}_1$.
Perhaps more surprising is that the Tukey type of $\mathcal{U}_1$ contains ultrafilters which are Rudin-Keisler incomparable.
For example,
it follows by arguments using the Abstract Ellentuck Theorem that $\mathcal{U}_1\cdot\mathcal{U}_1$ and $\mathcal{Y}_2$ are Rudin-Keisler incomparable.
\end{example}

From Theorem \ref{thm.TukeyU_1}, we obtain the analogue of Laflamme's result for the Rudin-Keisler ordering now in the context of Tukey types.

\begin{thm}\label{cor.1Tpred}
If $\mathcal{V}\le_T\mathcal{U}_1$,
then one of the following  must hold:
\begin{enumerate}
\item
$\mathcal{V}\equiv_T\mathcal{U}_1$;
or
\item
$\mathcal{V}\equiv_T\mathcal{U}_0$;
or
\item
$\mathcal{V}$ is a principal ultrafilter.
\end{enumerate}
\end{thm}

\begin{proof}
Let $\mathcal{V}$ be a  nonprincipal ultrafilter such that  $\mathcal{V}\le_T\mathcal{U}_1$.
Theorem \ref{thm.TukeyU_1} implies that $\mathcal{V}$ is isomorphic, and hence Tukey equivalent, to the ultrafilter on $\mathcal{S}$ generated by the  $\vec{\mathcal{W}}$-trees, where
 for each $s\in\hat{\mathcal{S}}\setminus\mathcal{S}$, the ultrafilter $\mathcal{W}_s$ is $\mathcal{Y}_{n(s)}$ for some $n(s)<\om$.
If all $n(s)=0$, then $\mathcal{V}$ is Tukey equivalent to $\mathcal{U}_0$.
Otherwise, for some $s$, $n(s)>0$.
In this case, Proposition \ref{prop.structureU_R(n)} and induction on the lexicographical rank of $\mathcal{F}$ imply that $\mathcal{V}$ is Tukey equivalent to $\mathcal{U}_1$.
\end{proof}

\bibliographystyle{amsplain}
\bibliography{referencesR_1}

\providecommand{\bysame}{\leavevmode\hbox to3em{\hrulefill}\thinspace}
\providecommand{\MR}{\relax\ifhmode\unskip\space\fi MR }
\providecommand{\MRhref}[2]{%
  \href{http://www.ams.org/mathscinet-getitem?mr=#1}{#2}
}
\providecommand{\href}[2]{#2}
\begin{thebibliography}{10}

\bibitem{Bartoszynski/JudahBK}
S.~Tomek Bartoszy{\'{n}}ski and Haim Judah, \emph{{S}et {T}heory on the
  {S}tructure of the {R}eal {L}ine}, A. K. Peters, Ltd., 1995.

\bibitem{Blass74}
Andreas Blass, \emph{Ultrafilter mappings and their {D}edekind cuts},
  Transactions of the American Mathematical Society (1974), no.~188, 327--340.

\bibitem{Dobrinen10}
Natasha Dobrinen, \emph{Continuous cofinal maps on ultrafilters}, Submitted.

\bibitem{Dobrinen/Todorcevic10}
Natasha Dobrinen and Stevo Todorcevic, \emph{Tukey types of ultrafilters},
  Illinois Journal of Mathematics (2012), 33 pp. To appear.

\bibitem{MR0349393}
Erik Ellentuck, \emph{A new proof that analytic sets are {R}amsey}, Journal of
  Symbolic Logic \textbf{39} (1974), 163--165.

\bibitem{Erdos/Rado50}
Paul Erd{\H{o}}s and R~Rado, \emph{A combinatorial theorem}, Journal of the
  London Mathematical Society \textbf{25} (1950), 249--255.

\bibitem{MR1644345}
Ilijas Farah, \emph{Semiselective coideals}, Mathematika \textbf{45} (1998),
  no.~1, 79--103.

\bibitem{Galvin68}
Fred Galvin, \emph{A generalization of {R}amsey's theorem}, Notices of the
  American Mathematical Society \textbf{15} (1968), 548.

\bibitem{Isbell65}
John Isbell, \emph{The category of cofinal types. {II}}, Transactions of the
  American Mathematical Society \textbf{116} (1965), 394--416.

\bibitem{Juhasz66}
Istvan Juh{\'{a}}sz, \emph{Remarks on a theorem of {B}.\
  {P}osp{\'{i}}{\v{s}}il}, General Topology and its Relations to Modern
  Analysis and Algebra, Academia Publishing House of the Czechoslovak Academy
  of Sciences, Praha, 1967, pp.~205--206.

\bibitem{Laflamme89}
Claude Laflamme, \emph{Forcing with filters and complete combinatorics}, Annals
  of Pure and Applied Logic \textbf{42} (1989), 125--163.

\bibitem{MR2330595}
Jos{\'e}~G. Mijares, \emph{A notion of selective ultrafilter corresponding to
  topological {R}amsey spaces}, Mathematical Logic Quarterly \textbf{53}
  (2007), no.~3, 255--267.

\bibitem{Proml/Voigt85}
Hans~J{\"{u}}rgen Pr{\"{o}}ml and Bernd Voigt, \emph{Canonical forms of
  {B}orel-measurable mappings ${\Delta}:[\om]^{\om}\rightarrow\mathbb{R}$},
  Journal of Combinatorial Theory, Series A \textbf{40} (1985), 409--417.

\bibitem{Pudlak/Rodl82}
Pavel Pudlak and Vojtech R{\"{o}}dl, \emph{Partition theorems for systems of
  finite subsets of integers}, Discrete Mathematics \textbf{39} (1982), 67--73.

\bibitem{Raghavan/Todorcevic11}
Dilip Raghavan and Stevo Todorcevic, \emph{Cofinal types of ultrafilters},
  Annals of Pure and Applied Logic (2011), To appear.

\bibitem{Ramsey29}
F.P. Ramsey, \emph{On a problem of formal logic}, Proceedings of the London
  Mathematical Society \textbf{30} (1929), 264--296.

\bibitem{TodorcevicBK10}
Stevo Todorcevic, \emph{Introduction to {R}amsey {S}paces}, Princeton
  University Press, 2010.

\end{thebibliography}

\end{document}